\documentclass[11pt]{article}

\usepackage{amsmath,amsthm}

\usepackage[left=2cm,right=2cm,top=2cm,bottom=2cm]{geometry}

\usepackage{amssymb,mathrsfs,mathtools}
\usepackage[
backend=biber,
style=alphabetic,
sorting=nyt,
backref=false,
url=false,
doi=false,
isbn=false,
giveninits=true]{biblatex}

\renewbibmacro{in:}{
  \ifentrytype{article}{}{\printtext{\bibstring{in}\intitlepunct}}
  }

  \DeclareFieldFormat*{title}{\mkbibemph{#1}}
  \DeclareFieldFormat*{journaltitle}{#1}
  \DeclareFieldFormat[article]{volume}{\mkbibbold{#1}}
  \DeclareFieldFormat[article]{number}{\bibstring{number}~#1}

  \renewbibmacro*{journal+issuetitle}{
    \usebibmacro{journal}
    \setunit*{\addspace}
    \iffieldundef{series}
      {}
      {\newunit
        \printfield{series}
    \setunit{\addspace}}
    \printfield{volume}
    \setunit{\addspace}
    \usebibmacro{issue+date}
    \setunit{\addcolon\space}
    \usebibmacro{issue}
    \setunit{\addcomma\space}
    \printfield{number}
    \setunit{\addcomma\space}
    \printfield{eid}
    \newunit
  }

  \AtEveryBibitem{\ifentrytype{book}{\clearfield{pages}}{}}
  
\DeclareNameAlias{author}{family-given}
\DeclareNameAlias{editor}{family-given}
\DeclareNameAlias{translator}{family-given}

\usepackage{bm}
\usepackage{enumitem}
\usepackage{tensor}
\usepackage{tikz-cd}
\usepackage{hyperref}
\hypersetup{
  colorlinks=true,
  linkcolor={red!50!black},
  filecolor=magenta,
  urlcolor=cyan,
  citecolor={blue!80!black}
}
\theoremstyle{plain}

\newtheorem{theorem}     [equation]  {Theorem}
\newtheorem{proposition} [equation]  {Proposition}
\newtheorem{lemma}       [equation]  {Lemma}

\newtheorem*{theorem*}               {Theorem}
\newtheorem*{proposition*}           {Proposition}
\newtheorem*{lemma*}                 {Lemma}
\newtheorem*{corollary*}             {Corollary}

\theoremstyle{definition}

\newtheorem{definition/} [equation]  {Definition}
\newenvironment{definition}{ \pushQED{\qed}\begin{definition/}} {\popQED\end{definition/}}

\newtheorem*{definition*/}           {Definition}
\newenvironment{definition*}{ \pushQED{\qed}\begin{definition*/}} {\popQED\end{definition*/}}

\theoremstyle{remark}

\newtheorem{remark/}     [equation]  {Remark}
\newenvironment{remark}{ \pushQED{\qed}\begin{remark/}} {\popQED\end{remark/}}
\newtheorem{example/}    [equation]  {Example}
\newenvironment{example}{ \pushQED{\qed}\begin{example/}} {\popQED\end{example/}}

\newtheorem*{remark*/}               {Remark}
\newenvironment{remark*}{ \pushQED{\qed}\begin{remark*/}} {\popQED\end{remark*/}}
\newtheorem*{example*/}              {Example}
\newenvironment{example*}{ \pushQED{\qed}\begin{example*/}} {\popQED\end{example*/}}

\newtheorem*{question*}              {Question}

\numberwithin{equation}{section}

\newenvironment{enumeratea}{
  \begin{enumerate}[label = (\alph*)]}{
  \end{enumerate}}

\newlist{pfparts}{description}{1}
\setlist[pfparts,1]{%
  font=\normalfont\textsf,
  itemindent=2pt,
  wide,
  itemsep=0pt,topsep=2pt,
  labelsep=0.75ex
}

\newcommand{\ol}[1]{\overline{#1}}

\newcommand{\cl}[1]{\mathcal{#1}}

\newcommand{\fk}[1]{\mathfrak{#1}}

\newcommand{\ti}[1]{\tilde{#1}}

\newcommand{\rra}{\rightrightarrows}

\newcommand{\todash}{\dashrightarrow}
\newcommand{\xar}[1]{\xrightarrow[]{#1}}
\newcommand{\xarl}[1]{\xleftarrow[]{#1}}
\newcommand{\fiber}[2]{\tensor[_{#1}]{\times}{_{#2}}}
\newcommand{\germ}[2]{[#1]_{{#2}}}
\newcommand{\sslash}{\mathbin{/\mkern-6mu/}}

\newcommand{\op}{\textrm{op}}

\newcommand{\define}[1]{\emph{#1}}

\newcommand{\cat}[1]{\mathsf{#1}}
\newcommand{\mor}[1]{\underline{\cat{#1}}}

\newcommand{\R}{\mathbb{R}}
\newcommand{\C}{\mathbb{C}}
\newcommand{\Z}{\mathbb{Z}}

\DeclareMathOperator{\Bis}{Bis}
\DeclareMathOperator{\colim}{colim}
\DeclareMathOperator{\dom}{dom}
\DeclareMathOperator{\Diff}{Diff}
\DeclareMathOperator{\Hol}{Hol}

\DeclareMathOperator{\GL}{GL}

\DeclareMathOperator{\Lan}{Lan}

\DeclareMathOperator{\pr}{pr}
\DeclareMathOperator{\Ran}{Ran}

\DeclareMathOperator{\SO}{SO}

\DeclareMathOperator{\U}{U}

\begin{document}





\title{Diffeological Riemannian orbifolds}
\author{David Miyamoto \footnote{Queen's University, Kingston, Canada.
    Email:
    \texttt{d.miyamoto@queensu.ca}}}




\date{\today}

\maketitle


\begin{abstract}
  We show that the data of a Riemannian metric on a differentiable stack presented by an orbifold groupoid is equivalent to the data of a Riemannian metric on its diffeological orbit space. As a consequence, we conclude that the classical notion of Riemannian orbifold is equivalent to that of a Riemannian diffeological orbifold. We use the framework for Riemannian diffeology introduced by Kuribayashi, Sakai, and Shiobara, and our result answers a problem they posed in the affirmative. More generally, we show that a Riemannian metric on a Lie groupoid, namely a 2-metric in the sense of del Hoyo and Fernandes, induces a Riemannian metric on its diffeological orbit space only if the Lie groupoid is regular, and that properness is a sufficient but not necessary condition.
\end{abstract}

\tableofcontents

\section{Introduction}
\label{sec:introduction}

The goal of this article is to compare two approaches to Riemannian metrics on singular spaces. This begs the definition of ``Riemannian metric'' and ``singular space.'' The former, in the category of smooth manifolds $\cat{Mfld}$, is unambiguous: a Riemannian metric on a manifold $M$ is a smooth and positive-definite section $g$ of the symmetric square of the cotangent bundle. Equivalently, $g$ is a smooth map $TM \fiber{}{M} TM \to \R$ that restricts to an inner product on each fibre $T_xM \times T_xM$. ``Singular space'' is a less precise term. We specialize to quotients of smooth manifolds, viewed as diffeological spaces or differentiable stacks.

The category of diffeological spaces $\cat{Dflg}$, introduced by Souriau \cite{Sour80}, is a quasi-topos that includes $\cat{Mfld}$ as a full subcategory. As such, every quotient of a manifold inherits a natural diffeology even when it does not inherit a manifold structure. For example, the following carry non-trivial diffeologies:
\begin{itemize}
\item The folded line $\R/\Z_2$, where $\Z_2$ acts by reflection.
\item The leaf space $M/\cl{F}$, where $M$ is the M\"{o}bius band $((0,1) \times \R)/((x,y) \sim (1-x,y+1))$ and $\cl{F}$ is the foliation induced by the vertical lines in $(0,1) \times \R$.
\item The quotients $\R^{2n}/\SO(2n)$ and $\R^{2n}/\U(n)$, where $\U(n)$ acts on $\R^{2n}$ by identifying $\R^{2n} \cong \C^n$.
\end{itemize}
All these spaces are homeomorphic, and in $\cat{Dflg}$ both $\R/\Z_2$ and $M/\cl{F}$ are isomorphic, and $\R^{2n}/\SO(2n)$ and $\R^{2n}/\U(n)$ are equal. However, $\R/\Z_2$ is not isomorphic to $\R^{2n}/\SO(2n)$ for any $n$. The folded line is an example of a \define{diffeological orbifold}, namely a diffeological space that is locally isomorphic to quotients $\R^n/\Gamma$, where $\Gamma$ is a finite subgroup of $\GL(n)$. We will call these locally-defined isomorphisms \define{charts}. The adjective ``diffeological'' is necessary because orbifolds predate diffeology: Satake \cite{Sat56} introduced orbifolds, then under the name V-manifolds, in the 1950s, as sets equipped with (an equivalence class of) atlases whose charts take values in the topological spaces $\R^n/\Gamma$. As shown in \cite{IglKarZad10}, the data of such a \emph{classical orbifold} is equivalent to the data of a diffeological orbifold.

We defer the definition of a diffeological space to Section \ref{sec:riemannian-orbifolds}; here it suffices to say that a diffeological space consists of a set $X$ and a collection of maps $p\colon U \to X$, called \define{plots}, whose domains range over all open subsets of all Euclidean spaces $\R^n$, $n\geq 0$. In the articles \cite{Igl23} and \cite{KurSakShiob25}, Iglesias-Zemmour and Kuribayashi, Sakai, and Shiobara, respectively, introduced Riemannian diffeological spaces. Iglesias-Zemmour's definition is:
\begin{definition*}
  A \define{Riemannian metric} on $X$ is an assignment to each plot $p\colon U \to X$ a symmetric non-negative covariant 2-tensor $\alpha_p \in \Gamma(\operatorname{Sym}^2(T^*U))$ such that, whenever $f\colon U \to V$ is smooth, and $q\colon V \to X$ is a plot such that $q \circ f = p$, we have $f^*(\alpha_q) = \alpha_p$. Furthermore, $\alpha$ must satisfy a definiteness condition: if $\gamma\colon \R \to X$ is a curve and $\alpha_\gamma(t)(1,1) = 0$ at some $t$, then for every diffeological 1-form $\omega$ (a notion we do not define here), $\omega_\gamma(t)(1) = 0$.
\end{definition*}
Note that the tensor $\alpha_p$ need not be a Riemannian metric on $U$. Kuribayashi, Sakai, and Shiobara's definition uses tangent functors. Observe that the domain of a Riemannian metric on $U$ is obtained by applying the endofunctor $\hat{T}_2 \coloneqq \hat{T} \fiber{}{1} \hat{T} \colon \cat{Eucl} \to \cat{Eucl}$, where $\hat{T}$ denotes the usual tangent functor on $\cat{Eucl}$. Writing $y\colon \cat{Eucl} \to \cat{Dflg}$ for the inclusion functor, we can take the left Kan extension of $y\circ \hat{T}_2$ along $y$ to obtain an endofunctor $T_2\colon \cat{Dflg} \to \cat{Dflg}$. Kuribayashi, Sakai, and Shiobara define:
\begin{definition*}
  A \define{weak Riemannian metric} on  $X$ is a smooth map $g\colon T_2X \to \R$ such that for each plot $p\colon U \to X$ the pullback $p^*g  \colon y\hat{T}_2U \to \R$ is symmetric ($p^*g(v,w) = p^*g(w,v)$) and non-negative ($p^*g(v,v) \geq 0$).
\end{definition*}

The data of a Riemannian metric in the sense of \cite{Igl23}, excluding the definiteness condition, is equivalent to the data of a weak Riemannian metric in the sense of \cite{KurSakShiob25}. This was proved in \cite[Proposition 3.10]{KurSakShiob25}, and we will establish (Proposition \ref{prop:two-definitions-coincide}) a natural isomorphism between the relevant spaces of tensors.

A \define{Riemannian diffeological orbifold} is a diffeological orbifold $X$ equipped with a weak Riemannian metric $g$ for which the pullbacks $p^*g$ are Riemannian metrics for each chart $p$. In \cite[{(P3)}]{KurSakShiob25}, the authors ask whether a Riemannian diffeological orbifold consists of the same data as a Riemannian classical orbifold. To answer this question, we work in the greater generality of differentiable stacks.

Differentiable stacks, henceforth ``stacks,'' can be intrinsically defined as categories fibred in groupoids satisfying sheaf and descent conditions. However, conveniently, each stack is also presented by a Lie groupoid, in the sense that there is an equivalence of (weak) 2-categories from the weak 2-category of Lie groupoids (whose 1-morphisms are principal bibundles) to the strict 2-category of stacks. We denote the stack corresponding to a Lie groupoid $\cl{G} \rra M$ by $[M \sslash \cl{G}]$.

The relevant Lie groupoids in the example above are the action groupoids $\R \rtimes \Z_2$, $\R^{2n} \rtimes \SO(2n)$, and $\R^{2n} \rtimes U(n)$, and the holonomy groupoid $\Hol(\cl{F})$. The stacks $[\R\sslash \R \rtimes \Z_2]$ and $[M\sslash\Hol(\cl{F})]$ are isomorphic, but the stacks $[\R^{2n}\sslash\R^{2n} \rtimes \SO(2n)]$ and $[\R^{2n}\sslash \R^{2n}\rtimes \U(n)]$ are not. The Lie groupoid $\R \rtimes \Z_2$ is an \define{orbifold groupoid}, namely an \'{e}tale and proper Lie groupoid.

In \cite{HoyLoj19}, del Hoyo and Fernandes defined Riemannian metrics on stacks. A \define{2-metric} on $\cl{G}$ is a Riemannian metric $\eta^2$ on the manifold of composable arrows $\cl{G} \fiber{s}{t} \cl{G}$ compatible, in some sense, with the groupoid multiplication. A 2-metric induces a Riemannian metric $\eta$ on the base $M$, which in turn defines a not-necessarily-smooth map $\eta^N \colon N_2\cl{F} \coloneqq N\cl{F} \fiber{}{M} N\cl{F} \to \R$, where $N\cl{F}$ is the (possibly singular) normal bundle to the singular foliation $\cl{F}$ of $M$ given by the $\cl{G}$-orbits. We declare two 2-metrics to be equivalent if they induce the same map $N_2\cl{F} \to \R$. Del Hoyo and Fernandes define:
\begin{definition*}
  A \define{Riemannian metric} on the stack $[M\sslash\cl{G}]$ is an equivalence class of a 2-metric on $\cl{G}$.
\end{definition*}

For an orbifold groupoid, the data of a 2-metric on $\cl{G}$ is equivalent to the data of a Riemannian metric on the classical orbifold $M/\cl{G}$. Since $M / \cl{G}$ is also a diffeological orbifold, we ask whether and, if so, in what sense, it is canonically a Riemannian diffeological orbifold. Let $\cl{G}$ be a Lie groupoid, and let $j\colon N^*\cl{F} \hookrightarrow T^*M$ denote the dual of the quotient map $TM \twoheadrightarrow N\cl{F}$. We declare (Definition \ref{def:induces}) that an equivalence class of 2-metrics $[\eta^2]$ on $\cl{G} \rra M$ \define{induces} a weak Riemannian metric $\ol{\eta}$ on $M/\cl{G}$ if $\eta^N$ is smooth and $\pi^*(\ol{\eta}) = j_*(\eta^N)$. Equivalently, there is a natural map $N_2\cl{F} \to T_2(M/\cl{G})$, and our condition is that the following diagram commutes:
\begin{equation*}
  \begin{tikzcd}
     N_2\cl{F} & \R \\
     T_2(M/\cl{G}). &
     \ar["\eta^N", from=1-1, to=1-2]
     \ar["\lambda", from=1-1, to=2-1]
     \ar["\ol{\eta}"', from=2-1, to=1-2]
  \end{tikzcd}
\end{equation*}
Our findings may be summarized thus:
\begin{enumerate}[label = (\alph*)]
\item (Proposition \ref{prop:regularity-is-necessary}) If $\cl{G} \rra M$ has orbits of different dimensions within the same component of $M$, then a Riemannian metric on $[M\sslash\cl{G}]$ cannot induce a weak Riemannian metric on $M/\cl{G}$.
\item (Theorem \ref{thm:regular-proper-induces-metric}) If $\cl{G} \rra M$ is a proper and regular Lie groupoid, then every Riemannian metric on $[M\sslash\cl{G}]$ induces a weak Riemannian metric on $M/\cl{G}$. Note that $M/\cl{G}$ is a diffeological orbifold.
\item (Theorem \ref{thm:orbifold-groupoid-induces-metric}) If $\cl{G} \rra M$ is an orbifold groupoid, then every Riemannian metric on $M/\cl{G}$ is induced by a Riemannian metric on $[M\sslash\cl{G}]$.
\end{enumerate}
Points (b) and (c) imply that for orbifold groupoids, the notions of Riemannian metric on the stack $[M\sslash\cl{G}]$, on the diffeological orbifold $M/\cl{G}$, and on the classical orbifold $M/\cl{G}$, are all equivalent. On the other hand, item (a) shows that (b) and (c) cannot generalize to all stacks. Moving beyond orbifolds, we observe (Proposition \ref{prop:metrics-descend-for-Q-groupoids}) that if $\Gamma$ is a countable group acting smoothly and freely on a manifold $M$, then $M \rtimes \Gamma$ also satisfies the conclusions of (b) and (c). For an example, we may take $M = \R$ and $\Gamma = \Z + \alpha\Z$ for irrational $\alpha$, and achieve a complete understanding of the weak Riemannian metrics on the irrational torus $\R / (\Z + \alpha\Z)$. It would be interesting to have a uniform description of those Lie groupoids $\cl{G}$ that satisfy the conclusions of (b) and (c). 

Riemannian metrics on singular spaces appear throughout the literature. Satake \cite{Sat57} defined Riemannian orbifolds in his second paper treating orbifolds. Shortly thereafter Reinhart \cite{Rein59} introduced Riemannian foliations, which are foliated manifolds $(M, \cl{F})$ equipped with a ``bundle-like'' metric on $M$ that is intended to represent a metric on the leaf space $M/\cl{F}$. Indeed, a bundle-like metric is precisely a 0-metric on $\Hol(\cl{F})$. Gallego, Gualandri, Hector, and Reventos \cite{GalGualHecRev89} axiomatized the Riemannian structure of $\Hol(\cl{F})$, and produced a first notion of Riemannian groupoid. In the language of \cite{HoyFer18}, this is a Lie groupoid equipped with a 1-metric. Pflaum, Posthuma, and Tang \cite{PflPosTan14} distinguished 0-metrics on Lie groupoids, calling them ``transversally invariant Riemannian metrics.'' We already mentioned that we use del Hoyo and Fernandes's notion of Riemannian groupoid (stack) \cite{HoyFer18,HoyLoj19}. Recently, Holtrop \cite{Hol25} introduced a more general notion of Riemannian groupoid.

Riemannian diffeology is a more recent development. Shiobara has several preprints advancing the subject \cite{Shiob26c,Shiob26b,Shiob26a}. Independently of the aforementioned papers \cite{Igl23,KurSakShiob25}, Goldammer and Welker \cite{GolWel26} gave another definition of Riemannian diffeological space. Its relation to the metrics considered here remains to be explored.

The paper is structured as follows. In Section \ref{sec:riemannian-orbifolds} we introduce orbifolds. This includes a review of diffeology, Lie groupoids and stacks. In Section \ref{sec:riemannian-metrics} we recall the definitions of Riemannian metrics on diffeological spaces and on stacks. In Section \ref{sec:induc-riem-metr} we formally define what it means for a Riemannian metric on a stack to induce a weak Riemannian metric on the corresponding diffeological orbit space, and establish the main results (a), (b), and (c) enumerated above.

\subsection{Acknowledgements}
\label{sec:acknowledgements}

I would like to thank Katsuhiko Kuribayashi and Yusuke Shiobara for helpful discussions explaining their work and for the encouragement to write this article. Thanks also to Yael Karshon for clarifications on classical orbifolds, and to Francois Ziegler for suggesting an improvement to the proof of Proposition \ref{prop:lambda-1-is-iso-for-proper-regular}.

\subsection{AI declaration}
\label{sec:ai-declaration}
No AI system was used in the creation of this article.

\section{Orbifolds}
\label{sec:riemannian-orbifolds}

\subsection{Diffeological orbifolds}
\label{sec:defin-an-orbif}

We review some diffeological prerequisites, and point to the standard reference \cite{Igl13} and its reprint \cite{Igl22} for details. For a category $\cat{C}$ and objects $c,c' \in \cat{C}$, we write $\mor{C}(c,c')$ for the set of morphisms from $c$ to $c'$. Let $\cat{Eucl}$ denote the category of all open subsets of Euclidean spaces, i.e., all open subsets $U \subseteq \R^n$ for all $n \geq 0$. A morphism $f \in \mor{Eucl}(U,V)$ is a $C^\infty$-smooth map $f\colon U \to V$. A \define{diffeological space} is a set $X$ and an assignment $\cl{D}$ that associates to each $U \in \cat{Eucl}$ a subset $\cl{D}_U \subseteq \underline{\cat{Set}}(U,X)$ such that:
\begin{description}[font=\normalfont]
\item[Concreteness:] For each $U$ and each $x \in X$, the constant map $r \mapsto x$ is in $\cl{D}_U$.
\item[Pre-sheaf:] For each $f \in \underline{\cat{Eucl}}(U,V)$, we have $f^*\cl{D}_V \subseteq \cl{D}_U$.
\item[Sheaf:] For each open cover $\{U_\alpha\}$ of $U$, if $f\in \mor{Set}(U,X)$ and $f|_{U_\alpha} \in \cl{D}_{U_\alpha}$ for each $\alpha$, then $f\in \cl{D}_U$. 
\end{description}
We call the elements of $\bigsqcup_{U \in \cat{Eucl}} \cl{D}_U$ the \define{plots} of $X$. A diffeology is determined by its plots, so we sometimes abuse notation and write $\cl{D}$ for the union $\bigsqcup_{U \in \cat{Eucl}} \cl{D}_U$. We may denote a plot by $(U,p)$.

\begin{remark}
  A diffeological space is equivalently defined as a concrete sheaf on the concrete site $\cat{Eucl}$. This is explained in \cite{BaezHof11}.
\end{remark}

A \define{smooth map} between diffeological spaces $(X, \cl{D}^X)$ and $(Y, \cl{D}^Y)$ is a map $f\colon X \to Y$ such that $f_*\cl{D}^X \subseteq \cl{D}^Y$. The category $\cat{Dflg}$ has diffeological spaces as objects, and smooth maps as morphisms. By forgetting the diffeology, we get the forgetful functor $|\cdot| \colon \cat{Dflg} \to \cat{Set}$. This has both a left and right adjoint, which arise from the two primitive examples of diffeologies.

\begin{example}
Let $X$ be a set. The \define{discrete} diffeology $\cl{D}^\circ$ is the smallest possible diffeology: $\cl{D}^\circ_U$ consists of the locally constant maps $U \to X$. The functor $X \mapsto (X, \cl{D}^\circ)$ is a left-adjoint to the forgetful functor. The \define{coarse} diffeology $\cl{D}^\bullet$ is the largest possible diffeology: $\cl{D}^\bullet_U \coloneqq \underline{\cat{Set}}(U,X)$. The functor $X \mapsto (X, \cl{D}^\bullet)$ is a right-adjoint to the forgetful functor.
\end{example}

We also have a functor $D\colon \cat{Dflg} \to \cat{Top}$, which assigns to $X$ the final topology determined by the plots. We call this the \define{D-topology}. A subset $A\subseteq X$ is D-open if and only if $p^{-1}(A)$ is open in $U$ for all plots $(U,p)$. The functor $D$ has a left-adjoint,\footnote{This assigns to $T$ the diffeology $\cl{D}_U \coloneqq \mor{Top}(U,T)$} but it does not have a right-adjoint.

\begin{example}
  A smooth manifold is a Hausdorff and second-countable topological space equipped with a maximal smooth atlas. A manifold $M$ has a canonical diffeology $\cl{D}_U^M \coloneqq \underline{\cat{Mfld}}(U,M)$. The functor $M \mapsto (M, \cl{D}^M)$ is full, faithful, and injective on objects. Thus diffeologically smooth maps between manifolds are the usual smooth ones. The D-topology of $(M, \cl{D}^M)$ is the underlying topology of $M$. 
\end{example}

Unlike the category of smooth manifolds, the category of diffeological spaces is a quasi-topos, i.e., complete, co-complete, and (locally) Cartesian closed. In practical terms, every product, subset, coproduct, quotient, and mapping space of diffeological spaces inherits a diffeology. We summarize these constructions. Let $\cat{C}$ be a small category, and let $F\colon \cat{C} \to \cat{Dflg}$ be a functor.

\begin{description}[font=\normalfont]
\item[Limits:] The underlying set of $\lim F$ is the subset of the product $\prod_{c \in \cat{C}} Fc$ consisting of those $\alpha = (\alpha_c)_{c \in \cat{C}}$ satisfying: for every morphism $f \colon c \to c'$, we have $Ff(\alpha_c) = \alpha_{c'}$. A map $p\colon U \to \lim F$ has the form $p = (p_c)_{c \in \cat{C}}$, and $p$ is a plot of $\lim F$ if and only if each $p_c\colon U \to Fc$ is a plot.

  For example, if $A$ is a subset of $X$, the \define{subset diffeology} of $A$ consists of those maps $p\colon U \to A$ such that $p\colon U \to X$ is smooth. If $X$ and $Y$ are diffeological spaces, the \define{product diffeology} on $X \times Y$ consists of those maps $(p,q)\colon U \to X \times Y$ such that $p$ and $q$ are plots of $X$ and $Y$, respectively.
  
\item[Colimits:] The underlying set of $\colim F$ is the quotient of the disjoint union $\coprod_{c \in \cat{C}} Fc$ by the equivalence relation generated by $v \sim Ff(v)$ for each morphism $f\colon c \to c'$ and each $v \in Fc$. Denote the quotient map by $\pi\colon \coprod_{c \in \cat{C}} Fc \to \colim F$. A map $p\colon U \to \colim F$ is a plot if and only if for each $r \in U$ there is an open neighbourhood $W$ of $r$, some $c \in \cat{C}$, and a plot $q\colon W \to Fc$ such that $p|_W = \pi \circ q$.

  For example, if $\cl{R}$ is an equivalence relation on $X$, the \define{quotient diffeology} on $X/\cl{R}$ consists of those maps $p\colon U \to X/\cl{R}$ that locally lift along the quotient $\pi\colon X \to X/\cl{R}$. If $X$ and $Y$ are diffeological spaces, the \define{coproduct diffeology} on $X \amalg Y$ consists of those maps $p\colon U \to X\amalg Y$ that are locally plots of either $X$ or $Y$.

  \item[Mapping spaces:] Given diffeological spaces $X$ and $Y$, the \define{functional diffeology} on $\mor{Dflg}(X,Y)$ consists of those maps $p\colon U \to \mor{Dflg}(X,Y)$ whose adjoint map $\hat{p}\colon U \times X \to Y$, given by $(r,x) \mapsto p(r)(x)$, is smooth. This makes $\mor{Dflg}(X,Y)$ an exponential object in $\cat{Dflg}$.

\end{description}

  We also review some distinguished morphisms. We call the isomorphisms \define{diffeomorphisms}. We call the regular monomorphisms \define{inductions}. These are injective smooth maps $\iota\colon A \to X$ that restrict to diffeomorphisms $A \to \iota(A)$, where $\iota(A)$ carries the subset diffeology. We call the regular epimorphisms \define{subductions}. These are surjective smooth maps $\pi\colon X \to Y$ for which the natural bijection $X/\pi \to Y$ is a diffeomorphism, where $X/\pi$ carries the quotient diffeology. Diffeomorphisms are equivalently defined as surjective inductions or injective subductions. See \cite[Chapter 2]{Bloh24b} for a proof that inductions and subductions are indeed regular (in fact, strong) monomorphisms and epimorphisms.

  \begin{remark}
    Diffeology uses ``induction / subduction,'' rather than ``immersion / submersion,'' because these are not equivalent notions in $\cat{Mfld}$. The map $\R \to \R^2$ given by $t \mapsto (t^3, t^2)$ is an induction that is not an immersion \cite{Jor82,KarMiyWat24}. The figure-eight $(-\pi/2, 3\pi/2) \mapsto \R^2$ given by $t \mapsto (2\cos(t), \sin(2t))$ is an injective immersion that is not an induction. The map $\R^2 \to \R$ given by $(x,y) \mapsto xy$ is a subduction that is not a submersion. However, every surjective submersion is a subduction.
  \end{remark}

   Given a non-empty set $X$ and set of maps $\cl{A} = \{p\colon U_p \to X\}$ whose images cover $X$, we get the surjective map $\pi\colon \coprod_{p \in \cl{A}} U_p \to X$. Each $U_p$ carries its canonical diffeology. The diffeology on $X$ \define{generated} by $\cl{A}$ is defined by declaring the natural map $\left(\coprod_{p\in \cl{A}} U_p\right) / \pi \to X$ to be a diffeomorphism. For instance, the discrete diffeology is generated by the constant maps $\R^n \to X$, and a manifold's canonical diffeology is generated by the (inverses of) the charts of an atlas.
  
For diffeological spaces $X$ and $E$, we write $\psi\colon X \todash E$ to denote a smooth map $\psi\colon A \to B$, where $A$ and $B$ are D-open subsets of $X$ and $E$ equipped with their subdiffeologies. A \define{transition} is a map $\psi\colon X \todash E$ that is a diffeomorphism $A \to B$. Now consider a collection of diffeological spaces $\cl{E}$. We say that $X$ is \define{modelled on} $\cl{E}$, and call $X$ an \define{$\cl{E}$-space}, if $X$ is locally diffeomorphic to members of $\cl{E}$. Precisely, for every $x\in X$, there is a space $E_x \in \cl{E}$ and a transition $\psi \colon X \todash E_x$ with $x \in \dom \psi$. We will call such $\psi$ a \define{chart}. We take the category $\cl{E}$-$\cat{Space}$ to be the full subcategory of $\cat{Dflg}$ whose objects are $\cl{E}$-spaces.

\begin{example}
  \label{ex:E-spaces}
  In each example below, we change the collection $\cl{E}$.
  \begin{description}[font=\normalfont]
  \item[$n$-manifolds:] $\{\R^n\}$-$\cat{Space}$ is isomorphic to the category of not-necessarily Hausdorff (nnH) and not-necessarily second-countable (nnSC) $n$-dimensional manifolds. The isomorphism sends an $\{\R^n\}$-space $X$ to the set $X$ equipped with the maximal atlas consisting of all the transitions $\psi\colon X \dashrightarrow \R^n$.
  \item[Manifolds:] $\{\R^n \mid n \geq 0\}$-$\cat{Space}$ is isomorphic to the category of nnH and nnSC manifolds which may have connected components of different dimensions.
  \item[Convenient manifolds:] For a locally convex topological vector space $E$, we declare $p$ to be a plot of $E$ if $\ell \circ p\colon U \to \R$ is smooth for all continuous linear functionals $\ell\colon E \to \R$. Let $\cl{E}$ denote the collection of all Hausdorff and Mackey complete locally convex topological vector spaces, each equipped with this diffeology. The category $\cl{E}$-$\cat{Space}$ is isomorphic to the category of nnH \define{convenient} manifolds as defined in \cite{KriegMic97}. See \cite[Lemma 2.5]{Kih23} or \cite[Section 4]{Miy25}.
  \end{description}

  We cannot recover the condition that $X$ is Hausdorff or second-countable, nor that the components of $X$ have arbitrary but uniform dimension, by varying $\cl{E}$ alone.
\end{example}

\begin{definition}
  A \define{diffeological $n$-orbifold} is a Hausdorff and second-countable diffeological space that is locally modelled on
  \begin{equation*}
    \cl{E}_{n\textrm{-orb}} \coloneqq \{\R^n / \Gamma \mid \Gamma \text{ is a finite subgroup of } \Diff(\R^n)\}.
  \end{equation*}
\end{definition}

Since $\Gamma$ is compact we can linearize the $\Gamma$-action: each $\R^n/\Gamma$ is itself locally modelled on $\{\R^n / \Gamma' \mid \Gamma' \text{ is a finite subgroup of }\GL(n)\}$. Therefore each diffeological orbifold is modelled on this subcollection.

\begin{remark}
  We highlight some similar spaces that have appeared in the literature.
  \begin{description}[font=\normalfont]
  \item[Quasifolds:] Diffeological $n$-quasifolds, introduced in \cite{IglPrat21}, are diffeological spaces modelled on
    \begin{equation*}
      \{\R^n / \Gamma \mid \Gamma \text{ is a countable group of affine transformations}\}. 
    \end{equation*}
 These include the irrational tori $T_\alpha \coloneqq \R/(\Z + \alpha \Z)$ (for $\alpha$ irrational). While the D-topology of $T_\alpha$ is trivial, its diffeology is quite rich.
  \item[Orthofolds:] \define{Orthofolds}, introduced \cite{GuerIgl26}, are diffeological spaces modelled on
    \begin{equation*}
      \{\R^n/G \mid n \geq 0, \ G \text{ is a subgroup of } O(n)\}.
    \end{equation*}
  \end{description}
\end{remark}

The reader might be more familiar with the classical approach, for which we follow \cite{Car22}. A classical orbifold consists of a Hausdorff and second-countable topological space $X$ and a distinguished collection of homeomorphisms $\cl{A} = \{X \supseteq A \xar{\psi} B \subseteq \R^n/\Gamma\}$ that serves as an atlas. Iglesias-Zemmour, Karshon, and Zadka \cite[Theorem 46]{IglKarZad10} proved that every classical orbifold has a natural diffeology, with which it is a diffeological orbifold, and conversely every diffeological orbifold arises from a classical orbifold.

\subsection{Stacky orbifolds}
\label{sec:orbifold-stacks}

Another approach to handling quotients of manifolds is via differential stacks. Our review will be brief, and we refer to the textbook \cite{MoerMrc03} and the paper \cite{Ler10} for details. A \define{Lie groupoid} consists of a category $\cl{G} \rra M$, a manifold structure on $M$, and a nnH and nnSC manifold structure on $\cl{G}$ such that: all the structure maps (source $s$, target $t$, inversion $i$, multiplication $m$, and unit $u$) are smooth, the source and target are submersions, and the source and target fibres are Hausdorff. We call the map $(t,s)\colon \cl{G} \to M \times M$ the \define{anchor} of $\cl{G}$. To notate an arrow $g \in \cl{G}$ with $(t,s)(g) = (y,x)$, we write $y \xarl{g} x$. We denote the multiplication of $z \xarl{h} y$ and $y \xarl{g} x$ by $hg$. For each $x \in M$ the \define{isotropy group} $G_x \coloneqq (t,s)^{-1}(x,x)$ is an embedded submanifold of $\cl{G}$, and it is a Lie group with this manifold structure.

The base of a Lie groupoid is partitioned by the subsets $O_x \coloneqq t(s^{-1}(x))$, called \define{orbits}. We denote the orbit space by $M/\cl{G}$, and write $\pi\colon M \to M/\cl{G}$ for the quotient map. Each orbit carries the structure of an injectively immersed submanifold, uniquely determined by requiring that $t\colon s^{-1}(x) \to O$ is a principal $G_x$-bundle. The partition $\cl{F}$ of $M$ into the connected components of the orbits is a \define{singular foliation}. In particular, the subset $T\cl{F} \coloneqq \bigsqcup_{L \in \cl{F}} TL$ of $TM$ is a smooth \define{generalized distribution}, i.e., for each $x \in M$, the fibre $T_x\cl{F} = T_xL$ is a vector subspace of $T_xM$, and $T\cl{F} = \{X_x \mid X \in \fk{X}(M) \text{ and } X(M) \subseteq T\cl{F}\}$.

 \begin{example}
The following examples are standard.
    \begin{description}[font=\normalfont]
    \item[Submersions:] A surjective submersion $\pi \colon M \to N$ leads to the submersion groupoid $M \fiber{}{\pi} M \rra M$, whose arrows are $y \xarl{(y,x)} x$. The orbits of $M \fiber{}{\pi} M$ are the fibres of $\pi$.
    \item[Lie groups:] A Lie group $G$ leads to a Lie groupoid $G \rra \{*\}$, whose arrows are $* \xarl{g} *$, and multiplication is the group multiplication.
    \item[Action groupoids:] If a Lie group $H$ acts smoothly on $M$ from the right, we form the action groupoid $M \rtimes H$. This has the structure $M \times H \rra M$, and arrows are $x \xarl{(x,h)} x \cdot h$. The multiplication is $(x,h)(x \cdot h, g) = (x,hg)$. The orbits of $M \rtimes H$ are the orbits of $H$.
  \end{description}
  We use pseudogroups to pass from diffeological orbifolds to Lie groupoids.
\end{example}
\begin{definition}
    A collection $\Psi$ of transitions $M \todash M$ is a \define{pseudogroup} if:
    \begin{itemize}
    \item $\Psi$ is closed under inversion, composition (defining $\dom(\psi' \circ \psi) \coloneqq \psi^{-1}(\dom \psi')$), and restriction of domain to open subsets;
    \item $\Psi$ contains $1_M$;
    \item If $\psi\colon M \todash M$ is a transition, and there is an open cover $\{U_\alpha\}$ of $M$ such that $\psi|_{U_\alpha} \colon U_\alpha \to M$ is in $\Psi$ for each $\alpha$, then $\psi \in \Psi$.
    \end{itemize}
    The \define{germ groupoid} $\Gamma(\Psi) \rra M$ has for arrows the germs of the elements of $\Psi$, which we denote by $\psi(x) \xarl{\germ{\psi}{x}} x$. The multiplication is germ composition. We equip $\Gamma(\Psi)$ with the unique nnH and nnSC manifold structure such that the maps $x' \mapsto \germ{\psi}{x'}$ are open embeddings. This make $\Gamma(\Psi)$ a Lie groupoid. The orbits of $\Gamma(\Psi)$ are the orbits of $\Psi$.
  \end{definition}
  
We are particularly interested in regular Lie groupoids.
\begin{definition}
    A Lie groupoid is \define{regular} if its anchor map has constant rank.
  \end{definition}
  \begin{lemma}
    \label{lem:regular-definition}
    A Lie groupoid is regular if and only if its underlying foliation is regular. 
  \end{lemma}
  \begin{proof}
    This follows from the observation that $\ker_g(Tt,Ts) = T_gG_{s(g)}$. Indeed,
    \begin{align*}
      (Tt,Ts)(\xi) = 0 &\text{ iff } \xi \in Ts^{-1}(x) \text{ and } Tt(\xi) = 0 \\
      &\text{ iff } \xi \in T_{s(g)} G_x, \text{ because } t\colon s^{-1}(s(g)) \to \cl{O} \text{ is } G_x\text{-principal}.
    \end{align*}
    Therefore $(t,s)$ has constant rank if and only if the map $x \mapsto \dim (G_x)$ is constant, which holds if and only if the map $x \mapsto \dim(O_x)$ is constant.
  \end{proof}

For a regular Lie groupoid, at each $x \in M$ we can find a \define{transversal} $T$ to $\cl{F}$ through $x$. This is an embedded submanifold $T \subseteq M$ such that $T_{x'}M = T_{x'}T \oplus T_{x'}\cl{F}$ for each $x' \in T$. We denote transversality by $T \perp \cl{F}$. The collection $\{\pi|_T \colon T \to M/\cl{G} \mid T \perp \cl{F}\}$ generates the quotient diffeology on $M/\cl{G}$.
  
  Just as a Lie group may act on a manifold, a Lie groupoid $\cl{H} \rra N$ may act along a map $b\colon P \to N$. The data of such a right action is a triple $(\cl{H}, b, \text{act})$, where
  \begin{equation*}
    \text{act}\colon P \fiber{b}{t} \cl{H} \to P, \quad \text{act}(p,h) = p \cdot h,
  \end{equation*}
  is smooth, the iterated action $(p\cdot h)\cdot g$ is well-defined if and only if $p\cdot (hg)$ is well-defined, in which case both are equal, and $p\cdot u(b(p)) = p$ for all $p$. We similarly define a left action of $\cl{G} \rra M$ along $a\colon P \to M$, where the action map has domain $\cl{G} \fiber{s}{a} P$. An action of $\cl{H}$ on $P$ is captured by the action groupoid $P \rtimes \cl{H}$, which has the structure $P \fiber{b}{t} \cl{H} \rra P$, and whose arrows are $p \xarl{(p,h)} p \cdot h$.

  Given a right action of $\cl{H}$ on $P$, we call a map $\pi\colon P \to B$ a \define{right-principal $\cl{H}$-bundle} if $\pi$ is a surjective submersion, the $P \rtimes \cl{H}$-orbits are subordinate to the fibres of $\pi$, and the anchor map for $P \rtimes \cl{H}$ is a diffeomorphism onto its image. In this case $\cl{H}$ acts on the fibres of $\pi$ freely and transitively. A \define{right-principal (rp) bibundle} between Lie groupoids $\cl{G} \rra M$ and $\cl{H} \rra N$ consists of a manifold $P$, two maps $a\colon P \to M$ and $b \colon P \to N$, and
  \begin{itemize}
  \item a right $\cl{H}$-action on $P$ along $b$, for which $a\colon P \to M$ is a principal right $\cl{H}$-bundle,
  \item a left $\cl{G}$-action of $P$ along $a$, for which the $\cl{G}$-orbits are subordinate to the $b$-fibres,
  \end{itemize}
  such that the actions commute (i.e., $(g \cdot p) \cdot h$ makes sense if and only if $g \cdot (p\cdot h)$ does, in which case they are equal). We will write $P\colon \cl{G} \to \cl{H}$, but a diagram better captures the structure:
  \begin{equation*}
    \begin{tikzcd}
      \cl{G} \circlearrowright & P & \circlearrowleft \cl{H} \\
      M & & N.
      \ar["a"', two heads, from=1-2, to=2-1]
      \ar["b", from=1-2, to=2-3]
    \end{tikzcd}
  \end{equation*}
  
  Given rp bibundles $P\colon \cl{G} \to \cl{H}$ and $Q\colon \cl{H} \to \cl{K}$, we can compose them to a rp bibundle $Q \circ P \colon \cl{G} \to \cl{K}$. However, this composition is not associative. Rather, it is associative up to isomorphism of rp bibundles. We call $P,P' \colon \cl{G} \to \cl{H}$ \define{isomorphic} if there is a diffeomorphism $\alpha \colon P \to P'$ that commutes with the groupoid and action structure maps. Adding rp bibundle isomorphism to the categorical structure yields a weak 2-category $\cat{Bi}$, consisting of Lie groupoids, rp bibundles, and isomorphisms of rp bibundles. We call the 1-isomorphisms in $\cat{Bi}$ \define{Morita equivalences}. Explicitly, a Morita equivalence is a rp bibundle $P \colon \cl{G} \to \cl{H}$ such that $b\colon P \to N$ is left $\cl{G}$-principal.

  Differentiable stacks, whatever they are, form a strict 2-category $\cat{St}$. There is an equivalence of (weak) 2-categories $\cat{Bi} \to \cat{St}$, which we denote on objects by $\cl{G} \mapsto [M\sslash\cl{G}]$. Therefore we may treat a stack as a Lie groupoid up to Morita equivalence. In particular, we view data on a Lie groupoid that is functorially preserved under Morita equivalence as data on the differentiable stack.

  To view orbifolds as stacks, we use Lie groupoids $\cl{G} \rra M$ that are:
  \begin{description}[font=\normalfont]
  \item[\'{e}tale:] $\dim(\cl{G}) = \dim(M)$, and
  \item[proper:] the anchor map is proper and $\cl{G}$ is Hausdorff. The orbit space $M/\cl{G}$ will be Hausdorff.
  \end{description}
  
\begin{definition}
  An \define{orbifold groupoid} is an \'{e}tale and proper Lie groupoid. A \define{stacky orbifold} is a stack presented by an orbifold groupoid. 
\end{definition}

\begin{example}
If $\Gamma$ is a finite group acting smoothly on $\R^n$, then the action groupoid $\R^n\rtimes \Gamma$ is an orbifold groupoid, hence presents the stacky orbifold $[\R^n\sslash\Gamma]$.  
\end{example}
\begin{remark}
  In \cite[Definition 3.24]{KarMiy25}, we defined an orbifold groupoid as a Lie groupoid $\cl{G} \rra M$ that is locally isomorphic to action groupoids $\R^n \rtimes \Gamma \rra \R^n$, where the $\Gamma$ are finite. By the linearization theorem for proper Lie groupoids (Proposition \ref{prop:linear-model}), every \'{e}tale and proper Lie groupoid is an orbifold groupoid in this sense. Conversely, these orbifold groupoids need not have Hausdorff orbit spaces.
\end{remark}

There is a functor $F\colon \cat{Bi} \to \cat{Dflg}$ defined by
\begin{description}[font=\normalfont]
\item[Objects:] $F(\cl{G}) = M/\cl{G}$;
\item[1-arrows:] $F(P)$ is the unique map $M/\cl{G} \to N/\cl{H}$ for which the diagram below commutes:
  \begin{equation*}
    \begin{tikzcd}
      M & P & N \\
      M/\cl{G} & & N/\cl{H}.
      \ar["a"', from=1-2, to=1-1]
      \ar["b", from=1-2, to=1-3]
      \ar[from=1-1, to=2-1]
      \ar[from=1-3, to=2-3]
      \ar["{F(P)}", from=2-1, to=2-3]
    \end{tikzcd}
  \end{equation*}
  \item[2-arrows:] $F$ is trivial on 2-arrows.
\end{description}
This is how we pass from Lie groupoids to diffeological spaces. To go backwards, given any Lie groupoid $\cl{G} \rra M$ we use the pseudogroup of \define{bisections}:
\begin{equation*}
  \Bis(\cl{G}) \coloneqq \{\tau \circ \sigma \colon M \todash M \mid \sigma \colon M \todash \cl{G} \text{ is a section of } s, \text{ and } \tau \circ \sigma \text{ is a transition}\}.
\end{equation*}
We also call $\sigma$ a bisection. When $\cl{G}$ is \'{e}tale, for each $g$ there is a unique bisection $\sigma_g$ such that $\sigma_g(s(g)) = g$. We call an \'{e}tale Lie groupoid \define{effective} when the map $\cl{G} \to \Gamma(\Bis(\cl{G}))$ given by $g \mapsto \germ{t\circ \sigma_g}{x}$ is injective.
\begin{proposition}[{\cite[Theorem 6.5]{KarMiy25}}]
  \label{prop:HS-orbifolds-equivalence}
  Let $X$ be a diffeological orbifold, and let $\cl{G}$ and $\cl{H}$ be effective orbifold groupoids.
  \begin{enumeratea}
  \item There is some effective orbifold groupoid $\cl{K}$ such that $F(\cl{K}) \cong X$.
  \item If $f\colon F(\cl{G}) \to F(\cl{H})$ is a diffeomorphism, then there is a Morita equivalence $P\colon \cl{G} \to \cl{H}$ such that $F(P) = f$.
  \item If $P,Q\colon F(\cl{G}) \to F(\cl{H})$ are Morita equivalences and $F(P) = F(Q)$, then $P$ and $Q$ are isomorphic.
  \end{enumeratea}
\end{proposition}

\begin{remark}
  Item (a) was shown in \cite{IglLaf18} for orbifolds, and \cite{IglPrat21} for quasifolds, respectively. Items (a) and (b) for orbifolds may also be derived from \cite[Theorem 5.32]{MoerMrc03}.

The full result of \cite{KarMiy25} allows for $X$ to be a quasifold, for $\cl{G}$ and $\cl{H}$ to be effective quasifold groupoids, for $f$ to be a local diffeomorphism, and for $P$ and $Q$ to be locally-invertible bibundles. In \cite{Miy26a}, we further enlarged the relevant objects, allowed for $f$ to be a local subduction, and allowed for $P$ and $Q$ to be submersive bibundles. Defining these terms will take us too far afield, but we mention that the objects include both quasifolds and leaf spaces of regular Riemannian foliations.
\end{remark}

\begin{example}
  \label{ex:build-orbifold-groupoid}
  It is instructive to illustrate (a). Take a diffeological $n$-orbifold $X$, and fix some countable collection $\{\psi_i\colon X \todash \R^n/\Gamma_i\}$ of transitions whose domains cover $X$. Let $\pi_i\colon \R^n \to \R^n/\Gamma_i$ denote the quotient map, and set $V_i \coloneqq \pi_i^{-1}(\text{image}(\psi_i))$. Take
  \begin{equation*}
    \Psi \coloneqq \{\text{transitions } f\colon V_i \todash V_j \mid \pi_j \circ f = \psi_j \circ \psi_i^{-1} \circ \pi_i\}.
  \end{equation*}
The collection $\Psi$ generates a pseudogroup on $\coprod_i V_i$, also denoted $\Psi$. The germ groupoid $\Gamma(\Psi)$ is an orbifold groupoid, and $\left(\coprod_i V_i\right) / \Gamma(\Psi) \cong X$. We will return to this example throughout the article.
\end{example}
\section{Riemannian metrics}
\label{sec:riemannian-metrics}

\subsection{Diffeological Riemannian metrics}
\label{sec:diff-riem-metr}

The notions of Riemannian metric introduced in \cite{Igl23} and \cite{KurSakShiob25} both use Kan extensions of certain functors. Write $y\colon \cat{Eucl} \to \cat{Dflg}$ for the standard categorical embedding. The \define{category of plots} of $X$ is the slice category $\cat{Plots}(X) \coloneqq y \downarrow X$. This has for objects the pairs $(U,p)$, where $p\colon yU \to X$ is a plot of $X$. A morphism of plots from $(U,p)$ to $(V,q)$ is a smooth map $f\colon U \to V$ such that $p = q \circ yf$. The category of plots of $X$ admits a functor $\Pi^X \colon \cat{Plots}(X) \to \cat{Eucl}$, which sends $(U,p)$ to $U$. Because $y$ is an embedding, we will often omit it from the notation.

To treat contravariant functors we use opposite categories. In particular, we have the opposite functor $y^\op\colon \cat{Eucl}^\op \to \cat{Dflg}^\op$. The slice category $X \downarrow y^\op$ is identical to the opposite category $\cat{Plots}(X)^\op$, and the functor $\Pi_X \colon X \downarrow y^\op \to \cat{Eucl}^\op$ is simply $(\Pi^X)^\op$.

Since $\cat{Eucl}$ is small, and $\cat{Dflg}$ is a quasi-topos that is locally small, the left and right Kan extensions of functors from $\cat{Eucl}$ to $\cat{Dflg}$ along $y$ always exist, and may be computed as follows.
\begin{description}[font=\normalfont]
\item[Left Kan ext.:] The left Kan extension of $F\colon \cat{Eucl} \to \cat{Dflg}$ along $y$ will be denoted $\Lan_y F$, and
  \begin{equation*}
  \Lan_y F (X) \coloneqq \colim(F \circ \Pi^X).
\end{equation*}
\item[Right Kan ext.:] The right Kan extension of $G\colon \cat{Eucl}^\op \to \cat{Dflg}$ along $y^\op$ will be denoted $\Ran_{y^\op} G$, and
\begin{equation*}
  \Ran_{y^\op} G(X) \coloneqq \lim(G \circ \Pi_X).
\end{equation*}

\end{description}
\begin{remark}
  Since $y$ and $y^\op$ are full and faithful, the natural unit and counit transformations $\eta^F \colon F \to \operatorname{Lan}_yF \circ y$ and $\epsilon^G \colon \Ran_{y^\op}G \circ y^\op \to G$ are natural isomorphisms.
\end{remark}

All our constructions originate with the (fibred) tangent functor.
\begin{definition}
  Let $\hat{T}_k\colon \cat{Eucl} \to \cat{Eucl}$ be the endofunctor
  \begin{equation*}
    \hat{T}_kU \coloneqq U \times \R^n \dots\times \dots \R^n \text{ when } U \subseteq \R^n, \quad \hat{T}_kf(u,v_1,\dots, v_k) \coloneqq (f(u), D_uf(v_1),\dots,D_uf(v_k)).
  \end{equation*}
\end{definition}
  We write $v = (v_1,\ldots, v_k)$. We denote $\hat{T} \coloneqq \hat{T}_1$. The bundle $\hat{T}^*U \to U$ is the usual cotangent bundle. 
\begin{definition}
  Let $\hat{\cl{T}}^k \colon \cat{Eucl}^\op \to \cat{Dflg}$ be the functor
  \begin{equation*}
    \hat{\cl{T}}^k U \coloneqq \Gamma\left(\bigotimes^k\hat{T}^*U\right), \quad \hat{\cl{T}}^kf \coloneqq f^*.
  \end{equation*}
We define $\cl{T}^k\colon \cat{Dflg}^\op \to \cat{Dflg}$ to $\cl{T}^k \coloneqq \Ran_{y^\op} \hat{\cl{T}}^k$. We call an element $\alpha$ of $\cl{T}^kX$ a \define{covariant $k$-tensor}.
\end{definition}
Computing the limit in $\cat{Dflg}$ yields:
\begin{lemma}
  The space of covariant $k$-tensors on $X$ is the subspace $\cl{T}^kX \subseteq \prod_{(U,p) \in \cat{Plots}(X)^\op} \hat{\cl{T}}^kU$ consisting of those $\alpha = (\alpha_p)$ such that $f^*\alpha_q = \alpha_p$ whenever $f \colon p \to q$ is a map of plots. Given a smooth map $f\colon X \to Y$, the pullback map $\cl{T}^kf \coloneqq f^*$ is $(f^*\beta)_p = \beta_{f \circ p}$.
\end{lemma}
  
We call $\alpha \in \cl{T}^kX$ \define{symmetric} if each $\alpha_p$ is symmetric, \define{alternating} if each $\alpha_p$ is alternating, and \define{non-negative} if each $\alpha_p$ is non-negative (meaning $\alpha_p(u,v_1,\dots, v_1) \geq 0$). For a plot $(U,p)$ and $(u,v) \in \hat{T}_kU$, we write $\alpha_p(u,v)$ for the pairing of $\alpha_p(u)$ with $v$. Alternating tensors are diffeological differential $k$-forms, and the space $\Omega^k(X)$ of these forms dates to the origin of diffeology. 

\begin{definition}[\cite{Igl23}]
  \label{def:Iglesias-Zemnmour-metric}
  A \define{Riemannian metric} on $X$ is a symmetric and non-negative covariant 2-tensor $g \in \cl{T}^2X$ satisfying the following definiteness condition: if $\gamma\colon \R \to X$ is a smooth curve, and $g_\gamma(t,1,1) = 0$ for some $t$, then for every 1-form $\omega \in \Omega^1(X)$, we have $\omega_\gamma(t,1) = 0$.
\end{definition}

For the definition in \cite{KurSakShiob25}, we need another functor.

\begin{definition}
  We define $T_k\colon \cat{Dflg} \to \cat{Dflg}$ to be $T_k \coloneqq \Lan_y y\hat{T}_k$.
\end{definition}

Computing the colimit in $\cat{Dflg}$ yields:

\begin{lemma}
  The space $T_kX$ is the quotient $T_kX \coloneqq \left(\coprod_{(U,p) \in \cat{Plots}(X)} y\hat{T}_kU\right)/{\sim}$, where ${\sim}$ is generated by declaring $(u,v) \sim \hat{T}_kf(u,v)$ for each map of plots $f\colon p \to q$ and each $(u,v)\in y\hat{T}^kU$.
\end{lemma}
Given a plot $(U,p)$ and $(u,v) \in y\hat{T}_kU$, we write $[u,v]_p$ for the element of $T_kX$ represented by $(u,v)$. This is also $Tp(u,v)$.
  \begin{definition}[{\cite[Definition 3.1]{KurSakShiob25}}]
    \label{def:weak-riemannian-metric}
    A \define{weak Riemannian metric} on $X$ is a map $g \colon T_2X \to \R$ such that, for each plot $p\colon U \to X$, the pullback $p^*g \colon y\hat{T}_2U \to \R$ given by $(u,v) \mapsto g([u,v]_p)$ is a symmetric and non-negative a covariant 2-tensor on $U$.
  \end{definition}

The following proposition relates the two definitions above, strengthening \cite[Proposition 3.5]{KurSakShiob25}.

\begin{proposition}
  \label{prop:two-definitions-coincide}
  There is a natural transformation $\gamma \colon \cl{T}^k \to \mor{Dflg}(T_k^\op-, \R)$ defined by
  \begin{equation*}
\gamma_X(\alpha)([u,v]_p) \coloneqq \alpha_p(u,v).
  \end{equation*}
  The map $\gamma_X$ is an induction. Its image consists of those smooth maps $g\colon T_kX \to \R$ such that for each plot $p\colon U \to X$, the pullback $p^*g$ is multilinear in $v$. Furthermore, $p^*\gamma_X(\alpha)$ is symmetric, respectively non-negative, if and only if $\alpha$ is symmetric, respectively non-negative.
\end{proposition}
In particular, $\gamma_X$ gives a bijection between the Riemannian metrics in the sense of \cite{Igl23} (without the definiteness condition), and those of \cite{KurSakShiob25}. A novel observation is that $\gamma_X$ is a smooth induction.
\begin{proof}
  See Appendix \ref{sec:prov-prop-refpr}.
\end{proof}

\begin{remark}
  \label{rem:other-definitions-of-metric}
Goldammer and Welker \cite[Definition 10]{GolWel26} give another notion of Riemannian metric. They use their own tangent functor $T^{\text{GW}}$, and define a Riemannian metric as a map $g\colon T^{\text{GW}}X \fiber{}{X} T^{\text{GW}}X \to \R$ that is symmetric and positive-definite on the fibres. It is not clear whether $T^{\text{GW}} \fiber{}{1} T^{\text{GW}}$ coincides with $T_2$. A comparison with this definition would be welcome, but we do not pursue it here.
  \end{remark}

  The non-degeneracy of a weak Riemannian metric is delicate to state. 
  
\begin{definition}[{\cite[Definition 3.7]{KurSakShiob25}}]
    We call a weak Riemannian metric $g$ on $X$ \define{definite with respect to a generating family} $\cl{A}$ if $p^*g$ is a usual Riemannian metric for each plot $p \in \cl{A}$. If $\cl{A}$ consists of a single map $\pi$, we will say $g$ is definite with respect to $\pi$.
  \end{definition}
By \cite[Proposition 3.10]{KurSakShiob25}, this definiteness condition implies that found in Definition \ref{def:Iglesias-Zemnmour-metric}.
  
  \begin{example}
    A weak Riemannian metric that is definite with respect to an atlas $\cl{A}$ of a manifold $M$ is a Riemannian metric in the usual sense. We make the same conclusion if $\cl{A}$ consists of immersions.
  \end{example}

  \begin{example}
    Let $M$ be a convenient manifold, as defined in \cite{KriegMic97}. The term ``weak Riemannian metric'' on $M$ already exists in the literature on infinite-dimensional manifolds, for instance \cite[Definition 4.1]{Sch23}. By \cite[Proposition 3.12]{KurSakShiob25}, the weak Riemannian metrics on $M$ in that sense are in bijection with our weak Riemannian metrics on the underlying diffeological space $M$ that are definite with respect to the atlas of $M$. This motivates the terminology.
  \end{example}

  Beyond manifolds, \cite{KurSakShiob25} give new examples of weak Riemannian metrics on adjunctions of mapping spaces. Our goal is find weak Riemannian metrics on quotients of manifolds.

  \begin{definition}
    A \define{Riemannian orbifold} is a diffeological $n$-orbifold $X$ equipped with a weak Riemannian metric that is definite with respect to the generating family
    \begin{equation*}
      \cl{A} \coloneqq \{\psi^{-1} \circ \pi \colon \pi^{-1}(\text{image}(\psi)) \to X \mid \psi\colon X \todash \R^n/\Gamma  \text{ is an orbifold chart}\}.
    \end{equation*}
  \end{definition}
  \begin{remark}
    \label{rem:riemannian-orbifold-generating-families}
    Take two charts $\psi\colon X \todash \R^n/\Gamma$ and $\psi'\colon X \todash \R^n/\Gamma'$ of a diffeological $n$-orbifold $X$. Let $V \coloneqq \pi^{-1}(\text{image}(\psi))$ and set $\varphi \coloneqq \psi^{-1} \circ \pi|_V$, and similarly take $V'$ and $\varphi'$. Let $g$ be a weak Riemannian metric on $X$. Given any $r \in V$ and $r' \in V'$ with $\varphi(r) = \varphi'(r')$, there is some smooth map $f\colon V \todash V'$ such that $f(r) = r'$ and $\varphi' \circ f = \varphi$. Moreover, any such map is necessarily a transition near $r$. This is a fundamental fact about orbifolds, cf., \cite[Corollary 2.15, Lemma 2.16]{KarMiy25}. The equality $\varphi^*g = f^*(\varphi')^* g$ implies that $(\varphi^*g)_r$ is definite if and only if $((\varphi')^*g)_{r'}$ is definite. Therefore if $g$ is definite with respect to any particular family of such $\varphi$, then $g$ is definite with respect to $\cl{A}$. 
  \end{remark}

  \begin{example}
    \label{ex:reflection-orbifold}
    Consider the folded line $\R / \Z_2$, where $\Z_2$ acts on $\R$ by reflection. Let $\eta$ denote the standard metric on $\R$. We can try to define a smooth map $\ol{\eta}\colon T_2(\R/\Z_2) \to \R$ by
    \begin{equation*}
      \ol{\eta}(T_2\pi(t,v_1,v_2)) \coloneqq \eta(t,v_1,v_2).
    \end{equation*}
    However, it is not clear that this is well-defined; our main result is that $\ol{\eta}$ is indeed well-defined, and $(X, \ol{\eta})$ is a Riemannian orbifold.
  \end{example}

It is easier to lift weak Riemannian metrics from quotient spaces.

  \begin{proposition}
    \label{prop:metric-on-orbifold-lifts}
    Suppose that $\cl{G} \rra M$ is an \'{e}tale Lie groupoid, and that $g$ is a weak Riemannian metric on $M/\cl{G}$ that is definite with respect to $\pi\colon M \to M/\cl{G}$. Then the pullback $\pi^*g$ is a Riemannian metric on $M$ that is invariant under $\Bis(\cl{G})$. If $\cl{G}$ is proper then $(M/\cl{G}, g)$ is a Riemannian orbifold.
  \end{proposition}
  \begin{proof}
By definition of definiteness, $\pi^*g$ is a Riemannian metric on $M$. For invariance, suppose that $t \circ \sigma \in \Bis(\cl{G})$. Then
    \begin{equation*}
      (t\circ \sigma)^* \pi^*g = (\pi \circ t \circ \sigma)^*g = \pi^*g.
    \end{equation*}
    When $\cl{G}$ is proper, it is an orbifold groupoid, and the definiteness of $\cl{G}$ with respect to $\pi$ implies definiteness with respect to the orbifold atlas $\cl{A}$ by Remark \ref{rem:riemannian-orbifold-generating-families}.
  \end{proof}

  In the next section, we will be able to recognize $\pi^*g$ as a 0-metric on $\cl{G}$, which is equivalent to the data of a Riemannian metric on the stack presented by $\cl{G}$.

\subsection{Riemannian metrics on stacks}
\label{sec:riem-metr-stacks}
  
We review the definition of a Riemannian metric on a groupoid or stack as given in \cite{HoyFer18,HoyLoj19}. This begins with the notion of a transverse metric. Let $p\colon E \to B$ be a surjective submersion, take $e \in E$, set $b \coloneqq p(e)$, and set $F \coloneqq p^{-1}(b)$. We have $\ker(T_ep) \cong T_eF$, so given a Riemannian metric $\eta$ on $E$, the map $T_ep$ restricts to an isomorphism $T_ep^\perp\colon T_eF^\perp \to T_bB$, where $\bullet^\perp$ denotes the orthogonal complement. We call $\eta$ \define{transverse with respect to} $p$, or simply \define{$p$-transverse}, if for every fibre $F$ and every $e,e' \in F$, the map
\begin{equation*}
(T_{e'}p^\perp)^{-1} \circ T_ep^\perp \colon T_eF^\perp \to T_{e'}F^\perp
\end{equation*}
  is an isometry. In this case $\eta$ induces a push forward metric $\eta_B \coloneqq p_*\eta$, and $T_ep^\perp$ is an isometry when $B$ is equipped with this metric.

  For a Lie groupoid $\cl{G} \rra M$, let $\cl{G}^{(2)} \coloneqq \cl{G} \fiber{s}{t} \cl{G}$ denote the manifold of composable arrows. By viewing elements of $\cl{G}^{(2)}$ as commutative triangles, we see it carries a natural action by the symmetric group $S_3$.

  \begin{definition}[{\cite[Definition 3.1.1]{HoyLoj19}}]
    A $2$-metric on a Lie groupoid $\cl{G}$ is a Riemannian metric $\eta^2$ on $\cl{G}^{(2)}$ that is transverse with respect to $\pr_1\colon \cl{G}^{(2)} \to \cl{G}$ and is invariant under the $S_3$-action.
  \end{definition}

  A 2-metric is also transverse with respect to $m$ and $\pr_2$. A 2-metric descends to a 1-metric on $\cl{G}$ and a 0-metric on $M$. We recall the latter. Fix a Lie groupoid $\cl{G} \rra M$ with underlying singular foliation $\cl{F}$. In the following, we use the framework of diffeological bundles of vector spaces (DBVs), which we review in Appendix \ref{sec:diff-bundl-vect}. Both $T\cl{F}$, equipped with its sub-diffeology, and the normal bundle $N\cl{F} \coloneqq TM / T\cl{F}$, equipped with its quotient diffeology, are DBVs. When $\cl{G}$ is regular, $T\cl{F}$ and $N\cl{F}$ are smooth vector bundles. For $v \in TM$, we denote its image in $N\cl{F}$ by $[v]$. The Lie groupoid $\cl{G}$ acts on $N\cl{F}$ along the anchor $\pi\colon N\cl{F} \to M$, by
    \begin{equation*}
    g \cdot [v] \coloneqq [Tt(\xi)], \text{ where } \xi \in T_g\cl{G} \text{ is chosen such that } [Ts(\xi)] = [v].
  \end{equation*}
  \begin{remark}
    To see that this action is smooth, we can use the auxiliary data of a connection on $\cl{G}$. This is vector bundle morphism $\Omega \colon s^*TM \to TG$, such that $(\pi, Ts) \circ \Omega = 1_{s^*TM}$ and $\Omega|_M = Tu$. Every Lie groupoid admits a connection. The action of $\cl{G}$ on $N\cl{F}$ fits in the diagram
    \begin{equation*}
      \begin{tikzcd}
        \cl{G} \fiber{s}{\pi} TM & TM \\
        \cl{G} \fiber{s}{\pi} N\cl{F} & N\cl{F},
        \ar["{Tt \circ \Omega}", from=1-1, to=1-2]
        \ar[two heads, from=1-1, to=2-1]
        \ar[two heads, from=1-2, to=2-2]
        \ar["\text{act}", from=2-1, to=2-2]
      \end{tikzcd}
    \end{equation*}
    because if $v \in T\cl{F}$, then writing $v = Ts(\Omega(g,v))$, we see that $Tt(\Omega(g,v))$ is also in $T\cl{F}$. Since the downwards arrows are subductions and the top arrow is smooth, the action is smooth.
  \end{remark}

  A Riemannian metric $\eta$ on $M$ induces, for each $x \in M$, a splitting $T_xM = T_x\cl{F} \oplus T_x\cl{F}^\perp$. These assemble into the DBV $T\cl{F}^\perp \coloneqq \bigsqcup_x T_x\cl{F}^\perp$. We have the map $I$ defined by:
  \begin{equation}
    \label{eq:definition-of-I}
    \begin{tikzcd}
      & TM & \\
      T\cl{F}^\perp & & N\cl{F}.
      \ar["\pi", two heads, from=1-2, to=2-3]
      \ar["\iota", hook, from=2-1, to=1-2]
      \ar["I", "\text{smooth bij.}"', from=2-1, to=2-3]
    \end{tikzcd}
  \end{equation}
While $I$ is a smooth bijective morphism of DBVs, it is not generally an isomorphism.

  \begin{proposition}
    \label{prop:when-I-is-a-diffeomorphism}
    Let $\cl{G}$ be a Lie groupoid over a connected base. The map $I\colon T\cl{F}^\perp \to N\cl{F}$ defined in \eqref{eq:definition-of-I} is a diffeomorphism if and only if $\cl{G}$ is regular.
  \end{proposition}
  \begin{proof}
   The converse is a standard fact, so assume that $I$ is a diffeomorphism. Let $\dim_{\cl{F}^\perp}$ denote the dimension map $x \mapsto \dim T_x\cl{F}^\perp$. Since $I^{-1}$ is smooth, the composition $I^{-1} \circ \pi$, viewed as a map $TM \to TM$, is a smooth bundle endomorphism of $TM$. The bundle $T\cl{F}^\perp$ arises as the image of this bundle endomorphism, and therefore $\dim_{\cl{F}^\perp}$ is lower-semicontinuous. On the other hand, since $TM = T\cl{F} \oplus T\cl{F}^\perp$, we also have $\dim_{\cl{F}^\perp} = \dim(M) - \dim_{\cl{F}}$, so $\dim_{\cl{F}^\perp}$ is also upper-semicontinuous. Therefore $\dim_{\cl{F}^\perp}$ is continuous, hence constant, and by Lemma \ref{lem:regular-definition}, $\cl{G}$ is regular.
\end{proof}

 The dual of $I$ is a smooth bijection $I^*\colon \bigotimes^2N^*\cl{F} \to \bigotimes^2T^*\cl{F}^\perp$. A Riemannian metric $\eta$ on $M$ restricts to a smooth section $\eta^\perp$ of $\bigotimes^2T^*\cl{F}^\perp$.
  
  \begin{definition}
    A $0$-\define{metric} on a Lie groupoid $\cl{G} \rra M$ is a Riemannian metric $\eta$ on $M$ such that the (not necessarily smooth) section $\eta^N \coloneqq (I^*)^{-1} \circ \eta^\perp$ of $\bigotimes^2 N^*\cl{F}$ is invariant under the $\cl{G}$ action.
  \end{definition}
  The $\cl{G}$ action on $\bigotimes^2 N^*\cl{F}$ is given by $(g\cdot \varphi)([v],[w]) \coloneqq \varphi(g\cdot [v], g\cdot [w])$.

  \begin{remark}
    This is equivalent to the definition of 0-metric in \cite[Definition 3.1]{HoyFer18} and \cite[Definition 3.12]{PflPosTan14}. Both those sources take an ``orbit-by-orbit'' approach, since for each orbit $O$ the normal bundle $NO$ is a smooth vector bundle. We use the entire, possibly singular, normal bundle $N\cl{F}$.
  \end{remark}
    \begin{remark}
    Notice the tension between $T\cl{F}^\perp$ and $N\cl{F}$: a Riemannian metric $\eta$ naturally restricts to a smooth section of $\bigotimes^2 T^*\cl{F}^\perp$, but $\cl{G}$ acts smoothly on $\bigotimes^2 N^*\cl{F}$. The result is that we must deal with the fact that $\eta^N$ may not be smooth.
  \end{remark}

A 2-metric $\eta^2$ on $\cl{G}$ induces a 0-metric by $\eta \coloneqq s_* (\pr_1)_* \eta^2$ (one must show that $(\pr_1)_*\eta^2$ is $s$-transverse). Conversely, there exist $0$-metrics that are not induced by any 2-metric. For \'{e}tale Lie groupoids, the situation is simpler: since the source $s\colon \cl{G} \to M$ and projection $\pr_1\colon \cl{G}^{(2)} \to \cl{G}$ are \'{e}tale, we can build a 2-metric inducing a given 0-metric by declaring $\eta^2 \coloneqq (\pr_1)^*s^*\eta$.
\begin{example}
  \label{ex:0-metric-on-etale-groupoid}
    For \'{e}tale $\cl{G} \rra M$, the leaves of $\cl{F}$ are points, so $N\cl{F} = TM$ and the $\cl{G}$-action on $TM$ is
    \begin{equation*}
      g \cdot v \coloneqq T(t \circ \sigma_g)(v), \textrm{ where } \sigma_g \textrm{ is the bisection with } \sigma_g(s(g)) = g.
    \end{equation*}
    It immediately follows that $\eta$ is a 0-metric on $\cl{G}$ if and only if $\eta = \eta^N$ is invariant under $\Bis(\cl{G})$.

    In particular, if $\cl{G} \rra M$ is the Lie groupoid built from a diffeological orbifold, as in Example \ref{ex:build-orbifold-groupoid}, then a $0$-metric on $M = \coprod_i V_i$ is a choice of Riemannian metric $\eta_i$ on each $V_i$, such that the collection $\{\eta_i\}$ is invariant under $\Bis(\cl{G}) = \{f \colon V_i \todash V_j \mid \pi_j \circ f = \psi_j \circ \psi_i^{-1} \circ \pi_i\}$. This is the usual notion of Riemannian metric on a classical orbifold, cf., \cite[Section 4.1]{Car22}.
 \end{example}

For stacks, we require a notion of metric that is invariant under Morita equivalence. 

 \begin{definition}[{\cite[Section 6]{HoyLoj19}}]
   Two 2-metrics $\eta^2_1$ and $\eta^2_2$ are \define{equivalent} if, denoting their induced $0$-metrics by $\eta_1$ and $\eta_2$, we have $\eta_1^N = \eta_2^N$.
 \end{definition}

 To justify this notion of equivalence, suppose that $P \colon \cl{G} \to \cl{H}$ is a Morita equivalence. The $\cl{G}$ and $\cl{H}$ actions combine into a left action of the product Lie groupoid $\cl{G} \times \cl{H}$ along $(a,b)\colon P \to M \times M$, given by $(g,h)\cdot p \coloneqq g\cdot p \cdot h^{-1}$. This yields the action Lie groupoid $(\cl{G} \times \cl{H}) \ltimes P$, and we have functors
 \begin{equation*}
   \begin{tikzcd}
     \cl{G} & (\cl{G} \times \cl{H}) \ltimes P & \cl{H},
     \ar["\phi"', from=1-2, to=1-1]
     \ar["\psi", from=1-2, to=1-3]
   \end{tikzcd}
 \end{equation*}
 where $\phi(g,h,p) = g$ and $\psi(g,h,p) = h^{-1}$. For brevity, we denote $\cl{K} \coloneqq (\cl{G} \times \cl{H}) \ltimes P$.

 Given a 2-metric $\eta^2$ on $\cl{G}$, we may construct \cite[Proposition 6.3.1]{HoyLoj19} a 2-metric $\ti{\eta}^2$ on $\cl{K}$ that is $\phi^{(2)}$-transverse and pushes down to $\eta^2$. Taking an appropriate average of $\ti{\eta}^2$, we may construct \cite[Proposition 6.3.2]{HoyLoj19} another 2-metric $(\ti{\eta}')^2$ that is equivalent to $\ti{\eta}^2$, is $\psi^{(2)}$-transverse, and pushes down to a 2-metric $P_*\eta^2$ on $\cl{H}$. It is proved in \cite[Theorem 6.3.3]{HoyLoj19} that $\eta \mapsto P_*\eta$ is a well-defined bijection on equivalence classes of 2-metrics. Examining that proof yields the following commuting diagram (in $\cat{Set}$):
 \begin{equation}
   \label{eq:transporting-2-metrics}
   \begin{tikzcd}
     N_2\cl{F}_{\cl{G}} & &  N_2\cl{F}_{\cl{K}} & & N_2\cl{F}_{\cl{H}} \\
     & \R & & \R. & \\
     \ar["{[T_2a]}"', from=1-3, to=1-1]
     \ar["{[T_2b]}", from=1-3, to=1-5]
     \ar["\eta^N"', from=1-1, to=2-2]
     \ar["{\ti{\eta}}^N", from=1-3, to=2-2]
     \ar["{(\ti{\eta}')}^N"', from=1-3, to=2-4]
     \ar["{(P_*\eta)^N}", from=1-5, to=2-4]
     \ar[equals, from=2-2, to=2-4]
   \end{tikzcd}
 \end{equation}
 Equivalently, $[T_2a] \colon (N_p\cl{F}_{\cl{K}}, {\ti{\eta}}^N) \to (N_{a(p)}\cl{F}_{\cl{G}}, \eta^N)$ is an isometry for each $p \in P$, and similarly for $[T_2b]$, and the metrics $\ti{\eta}$ and $\ti{\eta}'$ are equivalent.

 \begin{definition}[{\cite[Section 6]{HoyLoj19}}]
   A \define{Riemannian metric} on a stack is an equivalence class of 2-metrics on a Lie groupoid presenting the stack.
 \end{definition}
 
 \begin{remark}
   For an \'{e}tale Lie groupoid $\cl{G}$, since the underlying foliation is trivial, two 2-metrics are equivalent if and only if they induce the same 0-metrics. Furthermore, every 0-metric is induced by some 2-metric. Therefore a Riemannian metric on $[M\sslash \cl{G}]$ is entirely determined by a choice of 0-metric.
 \end{remark}

\section{Induced Riemannian metrics}
\label{sec:induc-riem-metr}

\subsection{Defining an induced metric}
\label{sec:defin-an-induc}

For a Lie groupoid $\cl{G} \rra M$, the orbit map $\pi$ gives the pullback $\pi^*\colon \cl{T}^k(M/\cl{G}) \to \cl{T}^kM$. From the other direction, the quotient map $TM \twoheadrightarrow N\cl{F}$ dualizes to an induction $j\colon N^*\cl{F} \hookrightarrow T^*M$, which gives the pushforward $j_*\colon \Gamma\left(\bigotimes^k N^*\cl{F}\right) \to \cl{T}^kM$. 

\begin{definition}
  \label{def:induces}
  A 0-metric $\eta$ on a Lie groupoid $\cl{G}\rra M$ \define{induces}, or \define{descends to}, a weak Riemannian metric $\ol{\eta}$ on $M/\cl{G}$ if $\eta^N$ is a smooth section of $\bigotimes^2N^*\cl{F}$ and $\pi^*(\ol{\eta}) = j_*(\eta^N)$.
\end{definition}

\begin{remark}
  Because $\pi$ is a subduction, the pullback $\pi^*\colon \cl{T}^k(M/\cl{G}) \to \cl{T}^kM$ is injective. This is proved for differential forms in \cite[Article 6.39]{Igl13}, but that proof works verbatim for general covariant tensors. Therefore a 0-metric $\eta$ induces at most one weak Riemannian metric on $M/\cl{G}$.
\end{remark}

\begin{example}
  Let $\cl{G}$ be an \'{e}tale Lie groupoid and let $g$ be a weak Riemannian metric on $M/\cl{G}$ that is definite with respect to $\pi\colon M \to M/\cl{G}$.  Proposition \ref{prop:metric-on-orbifold-lifts} says that $\pi^*g$ is a 0-metric on $\cl{G}$. Furthermore, triviality of $\cl{F}$ implies that $j$ is the identity, and so $\pi^*g = j_*(\pi^*g)$. Thus $\pi^*g$ is a 0-metric inducing $g$.
\end{example}

We can reformulate this definition by viewing metrics on $M/\cl{G}$ as maps $T_2(M/\cl{G}) \to \R$. Given a Lie groupoid $\cl{G} \rra M$, since the functor $T_k$ preserves fibre products along submersions, we have the Lie groupoid $T_k\cl{G} \rra T_kM$ whose structure maps are obtained by applying $T_k$ to the structure maps of $\cl{G}$.

\begin{lemma}
  \label{lem:defines-the-lambdas}
  Let $\cl{G} \rra M$ be a Lie groupoid. There exist unique maps $\lambda_1$,$\lambda_2$, and $\lambda_3$ such that the diagram below commutes (the maps $q$ are the obvious quotient maps):
  \begin{equation}
    \label{eq:defining-lambdas}
    \begin{tikzcd}
      & T_kM & \\
      & N_k\cl{F} & \\
      N_k\cl{F}/\cl{G} & T_kM/T_k\cl{G} & T_k(M/\cl{G}).
      \ar["q", from=1-2, to=2-2]
      \ar["q"', from=1-2, to=3-1]
      \ar["T_k\pi", from=1-2, to=3-3]
      \ar["q"', bend right, dashed, from=1-2, to=3-2]
      \ar["\lambda_2", from=2-2, to=3-2]
      \ar[from=2-2, to=3-1]
      \ar[from=2-2, to=3-3]
      \ar["\lambda_3", from=3-2, to=3-1]
      \ar["\lambda_1"', from=3-2, to=3-3]
    \end{tikzcd}
  \end{equation}
  Moreover, every map in Diagram (\ref{eq:defining-lambdas}) is a subduction, and $\lambda_3$ is a diffeomorphism.
\end{lemma}
\begin{proof}
  The proof is identical for each $k$, so we assume $k = 1$.  We treat $\lambda_2$ first. The quotient $q\colon TM \to N\cl{F}$ is the co-equalizer of the parallel morphisms $(\textrm{add}, \pr_1)$, as in the following diagram:
  \begin{equation*}
    \begin{tikzcd}
      TM \oplus T\cl{F} & TM & N\cl{F}.
      \ar["\textrm{add}", shift left, from=1-1, to=1-2]
      \ar["\pr_1"', shift right, from=1-1, to=1-2]
      \ar["q", from=1-2, to=1-3]
    \end{tikzcd}
  \end{equation*}
  But the quotient $q\colon TM \to TM/T\cl{G}$ also co-equalizes $(\textrm{add},\pr_1)$: indeed, if $v \in TM$ and $v_0 \in T\cl{F}$, then we may find $\xi \in \ker(Ts)|_M$ such that $Tt(\xi) = v_0$, and so $v+v_0 = Tt(Tu(v) + \xi)$ and $v = Ts(Tu(v) + \xi)$. By the universal property of a co-equalizer, $\lambda_2$ is well-defined as the unique map such that $q = \lambda_2 \circ q$.

  The argument for $\lambda_1$ is similar: the quotient map $q = \lambda_2 \circ q \colon TM \to TM/T\cl{G}$ is, by definition, the co-equalizer of $(Ts,Tt)$, but by functoriality of $T$, the map $T\pi$ also co-equalizes $(Ts,Tt)$. Therefore $\lambda_1$ is well-defined as the unique map such that $\lambda_1 \circ (\lambda_2 \circ q) = T\pi$.

  The argument for $\lambda_3$ is again similar: the quotient $q = \lambda_2 \circ q\colon TM \to TM/T\cl{G}$ is the co-equalizer of $(Ts,Tt)$, but the quotient map $q\colon TM \to N\cl{F}/\cl{G}$ also co-equalizes $(Ts,Tt)$. This is by definition of the $\cl{G}$ action on $N\cl{F}$: for every $\xi = \xi_g \in T\cl{G}$, we have $[Tt(\xi)] = g \cdot [Ts(\xi)]$, and so $q([Ts(\xi)]) = q([Tt(\xi)])$.  Therefore $\lambda_3$ is well-defined as the unique map such that $\lambda_3 \circ (\lambda_2 \circ q) = q$.

  Every quotient map $q$ is a subduction by definition of the quotient diffeology, and $T\pi$ is a subduction because $\hat{T}$ is a cosheaf \cite[Example 2.2.13, Proposition 2.2.16]{Bloh24b}. The universal property of subductions implies that all the $\lambda_i$ are subductions. To show that $\lambda_3$ is a diffeomorphism, it suffices to prove that it is injective. This is equivalent to $\lambda_2$ being $\cl{G}$-invariant. Let $[v] \in N_x\cl{F}$, take $g \in \cl{G}$ with $s(g) = x$, and take $\xi \in T_g\cl{G}$ with $Ts(\xi) = v$. Then
  \begin{equation*}
    \lambda_2(g \cdot [v]) = \lambda_2([Tt(\xi)]) = \lambda_2([Ts(\xi)]) = \lambda_2([v]).
  \end{equation*}
  This completes the proof.
\end{proof}

\begin{proposition}
  Let $\cl{G} \rra M$ be a Lie groupoid. A $0$-metric $\eta$ descends to a weak Riemannian metric $\ol{\eta}$ on $M/\cl{G}$ if and only if $\eta^N$ is a smooth section of $\bigotimes^2N^*\cl{F}$ and the following commutes:
  \begin{equation}
    \label{eq:induces-as-maps}
    \begin{tikzcd}
      N_2\cl{F} & \\
      T_2(M/\cl{G}) & \R.
      \ar["\eta^N", from=1-1, to=2-2]
      \ar["\ol{\eta}", from=2-1, to=2-2]
      \ar["\lambda_1 \circ \lambda_2"', from=1-1, to=2-1]
    \end{tikzcd}
  \end{equation}
\end{proposition}
\begin{proof}
  Consider the following diagram:
  \begin{equation*}
    \begin{tikzcd}
      \cl{T}^2(M/\cl{G}) & \cl{T}^2M & \Gamma\left(\bigotimes^2 N^*\cl{F}\right) \\
      \mor{Dflg}(T_2(M/\cl{G}),\R) & \mor{Dflg}(T_2M,\R) & \mor{Dflg}(N_2\cl{F}, \R) \\
      \mor{Dflg}(T_2M/T_2\cl{G},\R) & \mor{Dflg}(N_2\cl{F},\R).
      \ar["\cl{T}^k\pi", from=1-1, to=1-2]
      \ar["j_*"', from=1-3, to=1-2]
      \ar["\gamma_{M/\cl{G}}", from=1-1, to=2-1]
      \ar["\gamma_M", from=1-2, to=2-2]
      \ar["\cong", from=1-3, to=2-3]
      \ar["{(T_2\pi)^*}", from=2-1, to=2-2]
      \ar["q^*"', from=2-3, to=2-2]
      \ar["\lambda_1^*", from=2-1, to=3-1]
      \ar["\lambda_2^*", from=3-1, to=3-2]
      \ar["q^*", from=3-2, to=2-2]
      \ar[equals, from=3-2, to=2-3]
    \end{tikzcd}
  \end{equation*}
  The top two squares commute by definition of $\gamma_\bullet$, and the bottom square commutes by definition of the $\lambda_i$. Now, the map $q$ is surjective, therefore $q^*$ is injective, and the maps $\gamma_\bullet$ is are diffeomorphisms, and thus injective. The claim follows from diagram chasing.
\end{proof}

Recall that for an arbitrary 0-metric $\eta$, the map $\eta^N\colon N_2\cl{F} \to \R$ need not be smooth. One may ask whether it is nevertheless possible to find $\ol{\eta}$ such that the Diagram (\ref{eq:induces-as-maps}) commutes in $\cat{Set}$. We show that it is not. As a consequence, we lost no generality by assuming that $\eta^N$ is smooth in Definition \ref{def:induces}. 
\begin{proposition}
  \label{prop:regularity-is-necessary}
Suppose that $\cl{G} \rra M$ is a Lie groupoid over a connected base, with 0-metric $\eta$. Then there exists a smooth map $\ol{\eta} \colon T_2(M/\cl{G}) \to \R$ such that Diagram (\ref{eq:induces-as-maps}) commutes only if $\cl{G}$ is regular.
\end{proposition}

\begin{proof}
The commutativity of Diagram (\ref{eq:induces-as-maps}) implies that $\eta^N \colon N_2\cl{F} \to \R$ is smooth, and thus continuous in the D-topology. For $N_2\cl{F}$ this is the usual quotient topology.  Let $v_k \to v$ be a convergent sequence in $TM$ such that $v_k \in T\cl{F}$ for each $k$. Then, $[v_k] = 0 \in N\cl{F}$, and so by continuity and the fact $\eta^N$ restricts to a bilinear form on each fibre of $N\cl{F}$,
  \begin{equation*}
     \eta^N([v_k], [v_k]) = 0 \text{ for each }k \implies \eta^N([v],[v]) = 0.
   \end{equation*}
   Since $\eta^N$ is non-degenerate on the fibres of $N\cl{F}$, this means $[v] = 0 \in N\cl{F}$, or equivalently $v \in T\cl{F}$. Therefore $T\cl{F}$ is a closed smooth generalized distribution in $TM$.

   Smoothness of $T\cl{F}$ implies that the dimension map $x \mapsto \dim(T_x\cl{F})$ is lower-semicontinuous, so the set of points $A \subseteq M$ at which $\dim T_x\cl{F}$ is maximized (say, with value $k$) is open. We show it is closed, so let $x_k$ be a sequence of points in $A$ that converges to $x \in M$. Let $U \subseteq M$ be a coordinate chart of $M$ about $x$, giving a local trivialization $\Phi\colon TM|_U \to U \times \R^n$. Then $\{\pr_2(\Phi(T_{x_k}\cl{F}))\}_{k=1}^\infty$ is a sequence of $k$-planes in the Grassmannian $\operatorname{Gr}(k,n)$. The Grassmannian is compact, so by passing to a subsequence and re-labelling if necessary, we assume that this sequence converges to a $k$-plane $E \subseteq \R^n$.

   By definition of convergence in $\operatorname{Gr}(k,n)$, a vector $(x,v) \in U \times \R^n$ is in $\{x\} \times E$ if and only if it arises as a limit of a sequence $\{(x_k, v_k)\}$ with $v_k \in \Phi(T_{x_k}\cl{F})$. Since $T\cl{F} \subseteq TM$ is closed, this is equivalent to  $\Phi^{-1}(x,v) \in T_x\cl{F}$. Therefore $\Phi^{-1}(\{x\} \times E) = T_x\cl{F}$, and so $\dim(T_x\cl{F}) = k$, and $A$ is closed. Since $M$ is connected, $A = M$, and therefore $T\cl{F}$ is a regular subbundle of $TM$. By Lemma \ref{lem:regular-definition}, $\cl{G}$ is regular.  
\end{proof}

In light of Proposition \ref{prop:regularity-is-necessary}, we restrict our attention to regular Lie groupoids. In this case, recall that the identification $I\colon T\cl{F}^\perp \to N\cl{F}$ induced by a 0-metric $\eta$ (cf., Diagram \eqref{eq:definition-of-I}) is an isomorphism of smooth vector bundles, and $\eta^N\colon N_2\cl{F} \to \R$ is smooth. The next lemma states that, to test whether $\eta$ induces a weak Riemannian metric, it suffices to find a map $\ol{\eta}$ making Diagram (\ref{eq:induces-as-maps}) commute.

\begin{lemma}
  \label{lem:induced-map-is-weak-metric}
  Suppose that $\eta$ is a 0-metric on a regular Lie groupoid $\cl{G} \rra M$, and that there is some map $\ol{\eta} \colon T_2(M/\cl{G}) \to \R$ that makes Diagram (\ref{eq:induces-as-maps}) commute. Then $\ol{\eta}$ is a weak Riemannian metric on $M/\cl{G}$ that is definite with respect to $\{\pi|_T\colon T \to M/\cl{G} \mid T \perp \cl{F}\}$. 
\end{lemma}

\begin{proof}
Since $\lambda\coloneqq \lambda_2\circ \lambda_1$ is a subduction, the induced map $\ol{\eta} \colon T_2(M/\cl{G}) \to \R$ is smooth. Now let $p\colon U \to M/\cl{G}$ be a plot. Take $r \in U$ and $p$ locally along $\pi$ to a plot $q \colon (U,r) \dashrightarrow M$. Denote the quotient map $T_2M \to N_2\cl{F}$ by $\text{quot}$. Then
  \begin{align*}
    p^* \ti{\eta} &= q^* \pi^* \ol{\eta} \\
                  &= q^*(\ol{\eta} \circ T_2\pi) \\
                  &= q^*(\ol{\eta} \circ \lambda \circ \text{quot}) \\
                  &= q^*(\eta^N \circ \text{quot}),
  \end{align*}
  and $q^*(\eta^N \circ \text{quot})$ is a symmetric and non-negative 2-tensor. Thus $\ol{\eta}$ is a weak Riemannian metric.

  Now fix a transversal $\iota T\hookrightarrow M$ to $\cl{F}$. Reusing the previous computation,
  \begin{equation*}
    (\pi|_T)^*\ol{\eta} = \ol{\eta} \circ T_2\pi \circ T_2\iota =  (\eta^N \circ \text{quot}) \circ T_2\iota.
  \end{equation*}
  Since $T$ is transverse to $\cl{F}$, the map $T_2\iota \colon T_2T \to N_2\cl{F}$ is a fibrewise isomorphism. Therefore, since $\eta^N$ is definite, so is $(\eta^N \circ \text{quot}) \circ T_2\iota = (\pi|_T)^*\ol{\eta}$.
\end{proof}

The following is our main criterion for determining whether a 0-metric descends to the orbit space. 
\begin{proposition}
  \label{prop:lambda-1-iso-suffices}
  Suppose that $\eta$ is a 0-metric on a regular Lie groupoid $\cl{G}$, and that $\lambda_1\colon T_2M/T_2\cl{G} \to T_2(M/\cl{G})$ is an isomorphism. Then $\eta$ induces a weak Riemannian metric on $M/\cl{G}$ that is definite with respect to $\{\pi|_T\colon T \to M/\cl{G} \mid T \perp \cl{F}\}$.
\end{proposition}
\begin{proof}
  By the last clause in Lemma \ref{lem:defines-the-lambdas}, the map $\lambda_3$ is a diffeomorphism, or equivalently the fibres of $\lambda_2$ are the $\cl{G}$-orbits. Therefore the map $\eta^N$ corresponding to the 0-metric $\eta$ is constant on the fibres of $\lambda_2$. If $\lambda_1$ is an isomorphism, then $\eta^N$ is also invariant under the fibres of $\lambda_2\circ \lambda_1$, so it descends to a map $\ol{\eta}$ making Diagram (\ref{eq:induces-as-maps}) commute. The rest of the claims follow from Lemma \ref{lem:induced-map-is-weak-metric}. 
\end{proof}

To end this section, we give the stacky version of Definition \ref{def:induces}.

\begin{definition}
  \label{def:induced-from-stack}
  We say that a metric $[\eta^2]$ on a stack $[M\sslash \cl{G}]$ \define{induces} a weak Riemannian metric $\ol{\eta}$ on $M/\cl{G}$ if the 0-metric associated to any representative of $[\eta^2]$ induces $\ol{\eta}$ in the sense of Definition \ref{def:induces}. 
\end{definition}
Since two 2-metrics are defined to be equivalent if their 0-metrics define the same map $N_2\cl{F} \to \R$, the above notion is well-defined. However, a stack may be presented by different Morita equivalent Lie groupoids, so we should specify the sense in which our definition is invariant under Morita equivalence.

\begin{proposition}
  Suppose that $P\colon \cl{G} \to \cl{H}$ is a Morita equivalence and $\eta^2$ is a 2-metric on $\cl{G}$, with associated 0-metric $\eta$. Then $\eta$ induces a weak Riemannian metric on $M/\cl{G}$ if and only if $P_*\eta$, the 0-metric associated to $P_*\eta^2$, induces a weak Riemannian metric on $N/\cl{H}$.
\end{proposition}
\begin{proof}
 For brevity, we write $\cl{K}$ for the action groupoid $(\cl{G} \times \cl{H}) \ltimes P \rra P$. By the discussion at the end of Subsection \ref{sec:riem-metr-stacks} , we have the functors
  \begin{equation*}
   \begin{tikzcd}
     \cl{G} & (\cl{G} \times \cl{H}) \ltimes P & \cl{H},
     \ar["\phi"', from=1-2, to=1-1]
     \ar["\psi", from=1-2, to=1-3]
   \end{tikzcd}
 \end{equation*}
whose base maps are $a$ and $b$. Consider the following diagram: 
  \begin{equation*}
    \begin{tikzcd}
      & N_2\cl{F}_{\cl{K}} & \\
      N_2\cl{F}_{\cl{G}} & \R & N_2\cl{F}_{\cl{H}}\\
      T_2(M/\cl{G}) & & T_2(N/\cl{H}).
      \ar["{[T_2a]}"', from=1-2, to=2-1]
      \ar["{[T_2b]}", from=1-2, to=2-3]
      \ar["\lambda_1 \circ \lambda_2"', from=2-1, to=3-1]
      \ar["\lambda_1 \circ \lambda_2", from=2-3, to=3-3]
      \ar["T_2F(P)"', from=3-1, to=3-3]
      \ar["\eta^N", from=2-1, to=2-2]
      \ar["\ol{\eta}"', from=3-1, to=2-2]
      \ar["(P_*\eta)^N"', from=2-3, to=2-2]
      \ar["\exists\ \ol{P_*\eta} \ ?", from=3-3, to=2-2]
    \end{tikzcd}
  \end{equation*}
  The top triangle is Diagram \eqref{eq:transporting-2-metrics}. The outer pentagon commutes by naturality of $\lambda_1 \circ \lambda_2$ and functoriality of $T_2$. The left inner triangle commutes by definition of the induced metric $\ol{\eta}$. By Lemma \ref{lem:induced-map-is-weak-metric}, it suffices to show that $\ol{P_*\eta}$ exists. Recall that $P$ is a Morita equivalence, so by Proposition \ref{prop:HS-orbifolds-equivalence}, $F(P) \colon M/\cl{G} \to N/\cl{H}$ is a diffeomorphism, and thus so is $T_2F(P) \colon T_2(M/\cl{G}) \to T_2(N/\cl{H})$. Therefore we will define
  \begin{equation*}
    \ol{P_*\eta} \coloneqq \ol{\eta} \circ (T_2F(P))^{-1}.
  \end{equation*}
  We must check that the right inner triangle commutes. This is a diagram chase:
  \begin{align*}
    \ol{P_*\eta} \circ \lambda \circ [T_2b] &= \ol{\eta} \circ (T_2F(P))^{-1} \circ \lambda \circ [T_2b]  &&\text{by definition of } \widetilde{P_*\eta}\\
                                                                       &= \ol{\eta} \circ \lambda \circ [T_2a] &&\text{by definition of } F(P) \\
                                                                       &= \eta^N \circ [T_2a] &&\text{by definition of } \ol{\eta}\\
    &= (P_*\eta)^N \circ [T_2b] &&\text{by Diagram \eqref{eq:transporting-2-metrics}}.
  \end{align*}
We can cancel $[T_2b]$ on both sides because it is surjective. This proves that the right inner triangle commutes, and thus $P_*\eta$ does indeed induce $\ol{P_*\eta}$.
  
\end{proof}

\subsection{Finding induced metrics}
\label{sec:find-induc-metr}

In this last section, we show that if $\eta$ is a 0-metric on a regular and proper Lie groupoid, then it induces a Riemannian metric on the orbit space. Armed with Proposition \ref{prop:lambda-1-iso-suffices}, it suffices to prove that for such Lie groupoids, the map $\lambda_1$ is an isomorphism.

\begin{proposition}
  \label{prop:sufficient-condition-for-lambda-1-iso}
  Let $\cl{G} \rra M$ be a Lie groupoid. Suppose that, whenever we have a pair of parallel plots co-equalized by the quotient $\pi\colon M\to M/\cl{G}$, i.e.,
  \begin{equation*}
    \begin{tikzcd}
      U & M & M/\cl{G},
      \ar["p", shift left, from=1-1, to=1-2]
      \ar["q"', shift right, from=1-1, to=1-2]
      \ar["\pi", two heads, from=1-2, to=1-3]
    \end{tikzcd}
  \end{equation*}
  the plots $(T_kp, T_kq)$ are co-equalized by the quotient $T_kM \to T_kM/T_k\cl{G}$, i.e.,
  \begin{equation*}
    \begin{tikzcd}
      T_kU & T_kM & T_kM/T_k\cl{G}.
      \ar["T_kp", shift left, from=1-1, to=1-2]
      \ar["T_kq"', shift right, from=1-1, to=1-2]
      \ar[two heads, from=1-2, to=1-3]
    \end{tikzcd}
  \end{equation*}
Then the natural map $\lambda_1\colon T_kM/T_k\cl{G} \to T_k(M/\cl{G})$ is a diffeomorphism. 
\end{proposition}

\begin{remark}
  \label{rem:simpler-to-write-condition}
  To simplify notation, we write $M \fiber{}{\cl{G}} M$ for the fibre product $M \fiber{}{M/\cl{G}} M$. We can state Proposition \ref{prop:sufficient-condition-for-lambda-1-iso} as: if $(p,q)\colon U \to M\fiber{}{\cl{G}} M$ is a plot, then so is $(T_kp,T_kq)\colon T_kU \to T_kM \fiber{}{T_k\cl{G}} T_kM$.
\end{remark}

\begin{proof}
  Consider any map of plots $f\colon (U,p) \to (V,q)$, and take any $k$-tuple of vectors $v = (v_1,\ldots, v_k)$ from $T_rU$. Set $w \coloneqq (Tf(v_i))$, which is a $k$-tuple of vectors in $T_{f(r)}U'$. Lift $p$ along $\pi$ locally near $r$ to some map $\ti{p}$, and lift $q$ along $\pi$ locally near $f(r)$ to some map $\ti{q}$. Set $W \coloneqq \dom(p) \cap f^{-1}(\dom(\ti{q}))$. Then we have parallel morphisms
  \begin{equation*}
    \begin{tikzcd}
      W & M & M/\cl{G},
      \ar["\ti{p}", shift left, from=1-1, to=1-2]
      \ar["\ti{q} \circ f"', shift right, from=1-1, to=1-2]
      \ar["\pi", two heads, from=1-2, to=1-3]
    \end{tikzcd}
  \end{equation*}
  so by our assumption, $T_k\ti{p}$ and $T_k(\ti{q} \circ f)$ project to the same map in $T_kM/T_k\cl{G}$. In particular, $T_k\ti{p}(r,v)$ and $T_k\ti{q}(f(r),w)$ are in the same $T_k\cl{G}$-orbit.

It follows that if $(U,p)$ and $(V,q)$ are any two plots of $M/\cl{G}$, and $T_kp(r,v) = T_kq(s,w)$, then for any lifts $\ti{p}$ and $\ti{q}$ of $p$ and $q$ along $\pi$, both $T_k\ti{p}(r,v)$ and $T_k\ti{q}(s,w)$ are in the same $T_k\cl{G}$-orbit. Therefore, if we take $p = q = \pi\colon M \to M/\cl{G}$ (strictly speaking, we should pre-compose $\pi$ with a manifold chart) we find that if $T_k\pi(r,v) = T_k\pi(s,w)$, then $(r,v)$ and $(s,w)$ are in the same $T_k\cl{G}$-orbit. But this is exactly the statement that $\lambda_1$ is injective. Since $\lambda_1$ is also a subduction, it is a diffeomorphism.
\end{proof}

In order to apply Proposition \ref{prop:sufficient-condition-for-lambda-1-iso} to regular and proper Lie groupoids, we appeal to some general theory. First, proper Lie groupoids admit the following local models :
\begin{proposition}[{\cite[Corollary 3.11]{PflPosTan14}}]
  \label{prop:linear-model}
  Let $\cl{G} \rra M$ be a proper Lie groupoid, and let $x \in M$. Then there is a neighbourhood $U$ of $x$ in $M$, a neighbourhood $A$ of $x$ in $O_x$, a $G_x$-invariant neighbourhood $W$ of the origin in $N_xO_x$, and an isomorphism of groupoids
  \begin{equation*}
   \Phi \colon \cl{G}_U \to  \operatorname{Pair}(A) \times G_x \ltimes W.
  \end{equation*}
\end{proposition}

We also need the \define{canonical stratification}. This begins with the category $\cat{Rep}$, whose objects are smooth representations $H\circlearrowright V$ of Lie groups $H$ on finite-dimensional vector spaces $V$, and whose morphisms $H\circlearrowright V \to H'\circlearrowright V'$ are pairs $(\alpha, \lambda)$, where $\alpha \in \underline{\cat{Lie}}(H,H')$ and $\lambda \in \underline{\cat{Vect}}(V,V')$ and $\lambda(h \cdot v) = \alpha(h) \cdot \lambda(v)$. We declare
\begin{equation*}
  x \sim y \quad \text{if} \quad G_x \circlearrowright N_xO_x \underset{\cat{Rep}}{\cong} G_y\circlearrowright N_yO_y.
\end{equation*}
The representation $G_x\circlearrowright N_xO_x$ is the restriction of the $\cl{G}$ action on $N\cl{F}$. This is an equivalence relation on $M$, and we denote by $M_x^{\cong}$ the connected component of the equivalence class that contains $x$. The collection $\{M_x^{\cong}\mid x \in M\}$ the canonical stratification of $M$. When $\cl{G}$ is proper, this is a Whitney stratification, which entails several properties that we do not fully explain. We only state what we need.
\begin{proposition}[cf., {\cite[Proposition 26]{CrainMes18}}]
  Let $\cl{G} \rra M$ be a proper Lie groupoid. Each stratum $M_x^{\cong}$ is an embedded submanifold of $M$, and the collection $\{M_x^{\cong}\}$ is locally finite
\end{proposition}
We now continue towards our goal.
\begin{lemma}
  \label{lem:anchor-subduction-when-single-stratum}
    Let $\cl{G} \rra M$ be a proper Lie groupoid whose canonical stratification consists of a single stratum. Then the anchor map $(t,s) \colon \cl{G} \to M \times M$ is a subduction onto its image $M \fiber{}{\cl{G}} M$.
  \end{lemma}

  \begin{proof}
    Let $(p,q) \colon U_0 \to M\fiber{}{\cl{G}} M$ be a plot, and fix $r \in U_0$. We seek a plot $g \colon U_0 \dashrightarrow \cl{G}$ defined near $r$ that lifts $(p,q)$ along $(t,s)$, i.e., such that $t\circ g = p$ and $s\circ g = q$. Without loss of generality, we may assume that $p(r) = q(r)$. Indeed, if not, then choose a bisection $\sigma \colon M \dashrightarrow \cl{G}$ such that $p(r) \xleftarrow{\sigma(q(r))} q(r)$. Then given some lift $\ti{g}$ of $(p, t \circ \sigma \circ q)$, we get a lift of $(p,q)$ by taking $g \coloneqq \ti{g} \cdot (\sigma \circ q)$.

Set $x \coloneqq p(r) = q(r)$ and denote $H \coloneqq G_x$.  Linearizing with Proposition \ref{prop:linear-model}, take a neighbourhood $U$ of $x$, a neighbourhood $A$ of $x$ in $O_x$, an $H$-invariant neighbourhood $W$ of $N_xO_x$, and an isomorphism $\Phi\colon G_U \to \operatorname{Pair}(A) \times H \rtimes W$. In the model $\operatorname{Pair}(A) \times H \rtimes W$, observe that the normal representation at $(y,v)$ is isomorphic to the representation $H_v \circlearrowright T_vW/T_v(H\cdot v)$, where $H_v$ is the stabilizer of $v \in W$. Since these representations are all isomorphic by assumption, each $H_v$ is a subgroup of $H$ that is isomorphic to $H$. Since $H$ is compact, we conclude that $H_v = H$ for all $v$ (cf., \cite[Lemma 9]{CrainMes18}), and hence the $H$-action on $W$ is trivial. 

    Write $\Phi \circ p \coloneqq (p_1,p_2)$ and $\Phi \circ q \coloneqq (q_1,q_2)$. In the local model, the condition that $p$ and $q$ descend to the same map to $M/\cl{G}$ writes $(p_2,q_2) \colon U_0 \dashrightarrow W \fiber{}{G_x} W$. But since the $G_x$-action on $W$ is trivial, this means $p_2 = q_2$. Therefore the map
    \begin{equation*}
     \ti{g} \colon U_0 \dashrightarrow \operatorname{Pair}(A) \times G_x \rtimes W, \quad r' \mapsto (p_1(r'), q_1(r'), 1_x, p_2(r'))
    \end{equation*}
    is a lift of $(\Phi \circ p, \Phi \circ q)$ along $(t,s)$. Then $g \coloneqq \Phi^{-1}\circ \ti{g}$ is the desired lift of $(p,q)$ along $(t,s)$.

  \end{proof}

  \begin{proposition}
    \label{prop:lambda-1-is-iso-for-proper-regular}
    Let $\cl{G}$ be a proper and regular Lie groupoid. The natural map $\lambda_1\colon T_kM/T_k\cl{G} \to T_k(M/\cl{G})$ is an isomorphism. 
  \end{proposition}
  
 \begin{proof}
   The argument is the same for any value of $k$, so we assume $k=1$. By Proposition \ref{prop:sufficient-condition-for-lambda-1-iso}, it suffices to show that whenever $(p,q) \colon U \to  M\fiber{}{\cl{G}} M$ is a plot, so is $(Tp,Tq) \colon U \to TM \fiber{}{T\cl{G}} TM$. The map $(Tp,Tq) \colon TU \to TM \times TM$ is smooth, so we must show that its image is contained in $TM \fiber{}{T\cl{G}} TM$.

If we found some local lift $g \colon U \todash \cl{G}$  of $(p,q)$ along the anchor $(t,s)$, then
    \begin{equation*}
(Tp,Tq)(T\dom(g)) = (Tt,Ts)(Tg(TU)) \subseteq TM \fiber{}{T\cl{G}} TM.
    \end{equation*}
When $(t,s)$ is a subduction, the domains of all possible local lifts $g\colon U \todash \cl{G}$ cover $U$, and then we may conclude, as desired, that $(Tp,Tq)(TU) \subseteq TM\fiber{}{T\cl{G}} TM$. By Lemma \ref{lem:anchor-subduction-when-single-stratum}, the anchor $(t,s)$ is a subduction if $\cl{G}$ has a single stratum in its canonical stratification.

In general, the canonical stratification has many strata. At each $r \in U$, we seek a neighbourhood $W$ of $r$ such that $(Tp,Tq)(TW) \subseteq TM\fiber{}{T\cl{G}} TM$. Without loss of generality we may assume that $p(r) = q(r)$. Indeed, if not, take a bisection $\sigma \colon M \dashrightarrow \cl{G}$ such that $p(r) \xleftarrow{\sigma(q(r))} q(r)$. Then, supposing that we found the desired neighbourhood $W$ for the pair of plots $(Tp, T(t\circ \sigma \circ q))$, we observe that
\begin{equation*}
  TM \fiber{}{T\cl{G}} TM  \supseteq (Tp, T(t\circ \sigma\circ q))(TW) = (1 \times T (t\circ \sigma))(Tp, Tq)(TW),
\end{equation*}
and so the pre-image of $TM \fiber{}{T\cl{G}} TM$ under $1 \times T(t\circ \sigma)$ contains $(Tp,Tq)(TW)$. But the inverse map $(1 \times T(t\circ \sigma))^{-1}$ preserves $TM \fiber{}{T\cl{G}} TM$, so we have the desired containment for the pair of plots $(Tp,Tq)$.

Set $x \coloneqq p(r) = q(r)$. Since the canonical stratification is locally finite, we may take a neighbourhood $W$ of $r$ such that $p$ and $q$ meet only finitely many strata, say $S_1,\ldots, S_l$. The strata are saturated by the orbits, and $p$ and $q$ map into the same orbits, so setting $W_i \coloneqq p^{-1}(S_i)$, we also have $W_i = q^{-1}(S_i)$. Observe that $\{W_i\}_{i=1}^l$ is a finite cover of $W$ by locally closed subsets, since the $S_i$ are locally closed, so by \cite[Proposition (1.b)]{WatZieg26} (a version of a ``baby Baire'' property of $W$, cf.\ Remark \ref{rem:baby-baire}), the union $\bigcup_{i=1}^l W_i^\circ$ is dense in $W$. Now, each $G_{S_i} \rra S_i$ is regular, proper, and has a single stratum in its canonical stratification, so applying our observation at the beginning of the proof to the pair of restricted plots $(p,q) \colon W_i^\circ  \to S_i \fiber{}{\cl{G}_{S_i}} S_i$, we have
    \begin{equation*}
      (Tp,Tq)(TW_i^\circ) \subseteq TS_i \fiber{}{T\cl{G}_{S_i}} TS_i \subseteq TM \fiber{}{T\cl{G}} TM.
    \end{equation*}
    Repeating this observation for each $i$ yields
    \begin{equation*}
      (Tp,Tq)\left(\bigcup_{i=1}^l TW_i^\circ \right) \subseteq TM \fiber{}{T\cl{G}} TM.
    \end{equation*}
    Now we take the closure of both sides in $TM \times TM$. The right side is already closed because it is the image of $T(t,s)\colon T\cl{G} \to TM \times TM$, and by our assumption on $\cl{G}$, the anchor $(t,s)$ is a proper map of constant rank. On the left side, we have
    \begin{equation*}
      \overline{(Tp,Tq)\left(\bigcup TW_i^\circ \right)} = \overline{(Tp,Tq)\left(\overline{\bigcup TW_i^\circ} \right)} = \overline{(Tp,Tq)\left(TW \right)},
    \end{equation*}
    where the last equality follows from the density of $\bigcup W_i^\circ$ in $W$. This completes the proof.
  \end{proof}

  \begin{remark}
    \label{rem:baby-baire}
    The result \cite[Proposition (1.b)]{WatZieg26} that we used in the proof states: if $X$ is a topological space covered by a finite collection of locally closed sets, then the union of their interiors is dense. The same conclusion holds under replacing ``locally closed'' with ``closed'' [ibid.\ (1.a)], or ``finite'' by ``locally finite'' [ibid.\ (1.c)]. This may be regarded as a ``baby Baire'' property of $X$, because $X$ is Baire if and only if for every \emph{countable} cover by locally closed sets, the union of their interiors is dense [ibid.\ (3.b)]. Once again, we may replace ``locally closed'' with ``closed,'' yielding the classical Baire theorem [ibid.\ (3.a)], or ``countable'' with ``locally countable'' [ibid.\ (3.c)].
  \end{remark}
  
  \begin{remark}
    If $\cl{G}$ is not regular, then $T_kM/T_k\cl{G}$ need not be isomorphic to $T_k(M/\cl{G})$. For example, take $M \coloneqq \R^2$, and consider the linear action of $S^1 \cong SO(2)$ on $\R^2$. The fibre of $T_2\R^2/T_2S^1$ over the origin $[0]$ in $\R^2/S^1$ is the quotient of $\R^2 \times \R^2$ by the diagonal $S^1$-action $\theta \cdot (v,w) = (\theta \cdot v, \theta \cdot w)$. Thus, for instance, $[0,v,w] \neq [0,w,v]$ in $T_2\R^2/T_2S^1$ unless $v=w$. On the other hand, the reflection map $r(x,y) \coloneqq (y,x)$ preserves the fibres of the quotient $\pi \colon \R^2 \to \R^2/S^1$, and therefore for any $(0,v,w) \in T_2\R$ we have $T_2\pi(0,v,w) = T_2\pi(T_2r(0,v,w)) = T_2\pi(0,w,v)$, thus $(0,v,w)$ and $(0,w,v)$ always represent the same element of $T_2(\R^2/S^1)$.
  \end{remark}
\begin{theorem}
  \label{thm:regular-proper-induces-metric}
    Let $\cl{G}$ be a regular and proper Lie groupoid, and let $[\eta^2]$ be a Riemannian metric on $[M \sslash \cl{G}]$. Then $[\eta^2]$ induces a weak Riemannian metric on $M/\cl{G}$ that is definite with respect to $\{\pi|_T \colon T \to M/\cl{G} \mid T \perp\cl{F}\}$.
  \end{theorem}
  \begin{proof}
    This is a direct consequence of Definition \ref{def:induced-from-stack}, Proposition \ref{prop:lambda-1-is-iso-for-proper-regular}, and Proposition \ref{prop:lambda-1-iso-suffices}. 
  \end{proof}

  Note that the orbit space of a regular and proper Lie groupoid $\cl{G}$ is a diffeological orbifold, and the metric $\ol{\eta}$ on $M/\cl{G}$ induced by $\eta^2$ makes $(M/\cl{G}, \ol{\eta})$ a Riemannian orbifold.

  \begin{theorem}
    \label{thm:orbifold-groupoid-induces-metric}
    Let $\cl{G}$ be an orbifold groupoid. A Riemannian metric on $[M\sslash \cl{G}]$ descends to a weak Riemannian metric on $M/\cl{G}$ that is definite with respect to $\pi\colon M \to M/\cl{G}$. Conversely, a weak Riemannian metric $g$ on $M/\cl{G}$ that is definite with respect to $\pi$ pulls back to a 0-metric $\pi^*g$ on $\cl{G}$, which defines a metric on $[M\sslash \cl{G}]$ that induces $g$.  
  \end{theorem}
  \begin{proof}
    The first clause is a special case of Theorem \ref{thm:regular-proper-induces-metric}. The second follows from Proposition \ref{prop:metric-on-orbifold-lifts} and the fact that a 0-metric on an \'{e}tale Lie groupoid is always induced by a 2-metric.
  \end{proof}

    In \cite[Section 7.3, (P3)]{KurSakShiob25}, the authors pose the following problem. Let $X$ be a diffeological $n$-orbifold, and adopt the notation of Example \ref{ex:build-orbifold-groupoid}. Consider the category $\cat{C}$ whose objects are the plots $\varphi_i \coloneqq \psi_i^{-1} \circ \pi_i \colon V_i \to X$, and whose morphisms are commutative triangles
    \begin{equation}
      \label{eq:morphism-in-plot-category}
      \begin{tikzcd}
        V_i & & V_j\\
        & X. &
        \ar["\varphi_i"', from=1-1, to=2-2]
        \ar["\varphi_j", from=1-3, to=2-2]
        \ar["f", from=1-1, to=1-3]
      \end{tikzcd}
    \end{equation}
    This is a subcategory of $\cat{Plots}(X)$. Note that each $f$ is a local diffeomorphism (cf., \cite[Corollary 2.15]{KarMiy25}), thus locally $f$ coincides with elements of the pseudogroup $\Psi$. Hence a Riemannian  metric $\eta$ on $\coprod V_i$ that is invariant under $\Psi$, equivalently a 0-metric on $\Gamma(\Psi)$, induces a map  $\ti{\eta} \colon \colim_{\cat{C}} T_2V_i \to \R$. The authors of \cite{KurSakShiob25} ask: is there a weak Riemannian metric $\ol{\eta}$ lifting $\ti{\eta}$, as in the diagram below?
    \begin{equation*}
      \begin{tikzcd}
        \colim_{\cat{C}} T_2V_i & \R \\
        T_2X. &
        \ar["{\ti{\eta}}", from=1-1, to=1-2]
        \ar["\ol{\eta}"', from=2-1, to=1-2]
        \ar[from=1-1, to=2-1]
      \end{tikzcd}
    \end{equation*}

   We can answer ``yes.'' The colimit $\colim_{\cat{C}} T_2V_i$ is the quotient of $\coprod T_2V_i$ by the relation
    \begin{equation*}
      T_2V_i \ni (u,v,w) \sim T_2f(u,v,w) \in T_2V_j \text{ whenever } f \in \mor{C}(\varphi_i,\varphi_j).
    \end{equation*}
    But since each $f$ locally coincides with elements of $\Psi$, the classes of this relation are exactly the orbits of $T_2\Gamma(\Psi)$. In other words, $\colim_{\cat{C}}T_2V_i \cong \left(\coprod T_2V_i\right)/T_2\Gamma(\Psi)$. Therefore $\ol{\eta}$ exists if and only if the 0-metric $\eta$ induces a weak Riemannian metric on $X$, which is indeed the case by Theorem \ref{thm:orbifold-groupoid-induces-metric}. Furthermore, $\ol{\eta}$ is definite with respect to the quotient $\coprod V_i \to X$, so $(X, \ol{\eta})$ is a Riemannian diffeological orbifold.

  \subsection{Beyond orbifolds}
\label{sec:beyond-orbifolds}

A necessary condition for a 0-metric on $\cl{G}$ to induce a metric on $M/\cl{G}$ is that $\cl{G}$ is regular. A sufficient condition is that $\cl{G}$ is regular and proper. To close this article, we give another class of groupoids whose 0-metrics descend to their orbit space. We adopt the terminology from \cite[Definition 5.4]{Ahm24}.

\begin{definition}
  A \define{diffeological \'{e}tale map} is a smooth map $\pi\colon X \to Y$ that satisfies the following lifting condition: for every plot $p\colon U \to Y$, and every $r \in U$ and $x \in X$ with $p(r) = \pi(x)$, there exists a local lift $q\colon U \dashrightarrow X$ of $p$ along $\pi$ sending $r$ to $x$, and the germ of such a lift is unique.
\end{definition}

\begin{remark}
  If $M$ is a manifold, the surjective diffeological \'{e}tale maps $\pi\colon M \to Y$ were introduced earlier under the name ``Q-charts'' in \cite{Bar73}.  More generally, a surjective diffeologically \'{e}tale map $X \to Y$ is called a ``Q-chart'' in \cite[Definition 3.14]{Miy26a}. 
\end{remark}

\begin{example}
  If $\Gamma$ is a countable group acting smoothly and freely on a diffeological space $X$, then the quotient map $\pi\colon X \to X/\Gamma$ is diffeologically \'{e}tale. For a proof, see \cite[Example 2.7]{Miy25}, which is written for $X = \R$ and $\Gamma = \Z + \alpha \Z$ for $\alpha$ irrational, but generalizes \emph{mutatis mutandis}.
\end{example}

\begin{proposition}
  \label{prop:metrics-descend-for-Q-groupoids}
  Suppose that $\cl{G}$ is an \'{e}tale Lie groupoid whose orbit map $\pi\colon M \to M/\cl{G}$ is diffeological \'{e}tale. Then every 0-metric on $\cl{G}$ induces a weak Riemannian metric on $M/\cl{G}$ that is definite with respect to $\pi$.  Conversely, a weak Riemannian metric $g$ on $M/\cl{G}$ that is definite with respect to $\pi$ pulls back to a 0-metric $\pi^*g$ on $\cl{G}$ that induces $g$.  
\end{proposition}
\begin{proof}
  The converse holds for the exact same reason it holds in Theorem \ref{thm:orbifold-groupoid-induces-metric}, so we focus on the forward implication. By Proposition \ref{prop:lambda-1-iso-suffices} and Proposition \ref{prop:sufficient-condition-for-lambda-1-iso}, it suffices to show that if $(p,q)\colon U \to M \fiber{}{\cl{G}} M$ is a plot, then so is $(T_2p,T_2q)\colon T_2U \to T_2M\fiber{}{T_2\cl{G}} T_2M$, so take such a plot $(p,q)$. Exactly as in the first part of the proof of Proposition \ref{prop:lambda-1-is-iso-for-proper-regular}, which used only the fact that $\cl{G}$ is a Lie groupoid, we may work locally near $r \in U$, and assume that $p(r) = q(r) = x$.

  Our assumption on $(p,q)$ is equivalent to $\pi \circ p = \pi\circ q$. Since $\pi$ is diffeological \'{e}tale, this means that the germs of $p$ and $q$ are identical at $r$. But then, at least for $v,w \in T_rU$, we must have $T_2p(r,v,w) = T_2q(r,v,w)$. In other words, $(T_2p,T_2q)$ maps $T_rU \times T_rU$ into the diagonal in $T_2M\times T_2M$, and in particular into $T_2M\fiber{}{T_2\cl{G}} T_2M$. Observing that this holds for each $r$ completes the proof.
\end{proof}

\begin{example}
  Take $\Gamma_\alpha$ to be the subgroup $\Z + \alpha \Z$ of $\R$, for $\alpha$ irrational. The action groupoid $\R \rtimes \Gamma_\alpha$ satisfies the preconditions of Proposition \ref{prop:metrics-descend-for-Q-groupoids}. Therefore any 0-metric on $\R \rtimes \Gamma_\alpha$ descends to a weak Riemannian metric on $T_\alpha = \R / \Gamma_\alpha$ that is definite with respect to $\R \to \R / \Gamma_\alpha$. Since the orbits of $\Gamma_\alpha$ are dense in $\R$, and $\Gamma_\alpha$ acts by translation, the only 0-metric on $\R \rtimes \Gamma_\alpha$ (up to scaling) is the standard Euclidean metric. By Proposition \ref{prop:metrics-descend-for-Q-groupoids} again, the only weak Riemannian metric on $T_\alpha$ (up to scaling) is the one induced by the standard Euclidean metric. 
\end{example}
\pagebreak
\appendix

\section{Two diffeological generalities}
\label{sec:diff-tens}

\subsection{Proving Proposition \ref{prop:two-definitions-coincide} }
\label{sec:prov-prop-refpr}
We take the set-up directly from Proposition \ref{prop:two-definitions-coincide}.
\begin{proof}
  We verify each assertion directly, although some of them (well-definedness and naturality of $\gamma$) follow from the universal property of the right Kan extension.
  \begin{pfparts}
  \item[Well-defined:] Fix $X$, take a map of plots $f\colon p \to q$, and take $(u,v) \in y\hat{T}_kU$. Then since $\alpha \in \cl{T}^k$,
    \begin{equation*}
      \alpha_p(u,v) = (f^*\alpha_q)(u,v) = \alpha_q(\hat{T}_kf(u,v)).
    \end{equation*}
    Thus the definition of $\gamma_X(\alpha)$ does not depend on the representative of $[u,v]_p$, and $\gamma_X(\alpha)$ is a well-defined map of sets $T_kX \to \R$. Now we show that $\gamma_X(\alpha)$ is smooth. Since $T_kX$ is a quotient of a coproduct, it suffices to show that, for each plot $(U,p)$, the map
    \begin{equation*}
  y\hat{T}_kU \to \R, \quad (u,v)_p \mapsto \alpha_p(u,v)
    \end{equation*}
    is smooth. But this is true because $\alpha_p$ is a smooth section of $\bigotimes^k \hat{T}^*U$.

    Thus far, we have proved that $\gamma_X$ is a well-defined map of sets $\cl{T}^kX \to \mor{Dflg}(T_kX, \R)$, Now we show that $\gamma_X$ is smooth. In fact, we can simultaneously take a step toward proving it is an induction. Let $a = (a_p)\colon W \to \cl{T}^kX$ be a map. The composite $\gamma_X \circ a$ is smooth if and only if
    \begin{equation*}
      W \times T_kX \to \R, \quad (r,[u,v]_p) \mapsto \gamma_X(a(r))([u,v]_p) = a(r)_p(u,v) = a_p(r)(u,v)
    \end{equation*}
    is smooth. Once again, since $T_kX$ carries the quotient diffeology, this is smooth if and only if, for each plot $(U,p)$ of $X$, the map
    \begin{equation*}
      W \times y\hat{T}_kU \to \R, \quad (r,(u,v)) \mapsto a_p(r)(u,v)
    \end{equation*}
    is smooth. This is equivalent to the map $W \to \bigotimes^k T^*U$ given by $r \mapsto a_p(r)$ being smooth. But this is precisely what it means for $a$ to be a plot of $\cl{T}_kX = \lim(\hat{\cl{T}}^k \circ \Pi_X)$. Thus $\gamma \circ a$ is smooth if and only if $a$ is smooth.

    For naturality of $\gamma$, given a smooth map $f\colon X \to Y$, we must check that the diagram below commutes:
    \begin{equation*}
      \begin{tikzcd}
        \cl{T}^kY & \cl{T}^k X \\
        \mor{Dflg}(T_kY,\R) & \mor{Dflg}(T_kX,\R).
        \ar["\cl{T}^kf", from=1-1, to=1-2]
        \ar["\gamma_Y", from=1-1, to=2-1]
        \ar["\gamma_X", from=1-2, to=2-2]
        \ar["{(T_kf)^*}", from=2-1, to=2-2]
      \end{tikzcd}
    \end{equation*}
    This is a straightforward computation, and follows from the definitions of $\cl{T}^kf$ and $T_kf$. Indeed, fix $\beta \in \cl{T}^kY$, and $[u,v]_p \in T_kX$.
    \begin{align*}
      (\gamma_X \circ \cl{T}^kf(\beta))([u,v]_p) &= (\cl{T}^kf \beta)_p(u,v) = \beta_{fp}(u,v). \\
      ((T_kf)^* \circ \gamma_Y(\beta))([u,v]_p) &= \gamma_Y(\beta)(T_kf([u,v]_p)) = \gamma_Y(\beta)([u,v]_{fp}) = \beta_{fp}(u,v).
    \end{align*}

  \item[Induction:] Given a map $a\colon W \to \cl{T}^kX$, we already saw that $\gamma_X \circ a$ is a plot of $\mor{Dflg}(T_kX,\R)$ if and only if $a$ is a plot of $\cl{T}^kX$. It remains to show that $\gamma_X$ is injective. But this is clear: if $\gamma_X(\alpha) = \gamma_X(\beta)$, then for each plot $(U,p)$ of $X$, the sections $\alpha_p$ and $\beta_p$ must agree, which means that $\alpha = \beta$.

  \item[Image:] It follows immediately from the definition of $\cl{T}^k$ that, fixing $\alpha \in \cl{T}^kX$ and a plot $(U,p)$ of $X$, the map $v \mapsto \alpha_p(u,v)$ is multilinear in $v$. Conversely, suppose that $g\colon T_kX \to \R$ is a smooth map, and that for each plot $(U,p)$ of $X$, the map $(u,v) \mapsto g([u,v]_p)$ is multilinear in $v$. Define $\alpha_p(u,v) \coloneqq g([u,v]_p)$. Since the map $(u,v) \mapsto \alpha_p(u,v)$ is smooth and multilinear in $v$, the section $u \mapsto \alpha_p(u)$ is a smooth section of $\bigotimes^k\hat{T}^*U$. If $f\colon p \to q$ is a map of plots, and $(u,v) \in \hat{T}_kU$, then 
    \begin{equation*}
      f^*(\alpha_q)(u,v) = \alpha_q(\hat{T}_kf(u,v)) = g([\hat{T}_kf(u,v)]_q) = g([u,v]_p) = \alpha_p(u,v).
    \end{equation*}
    So $f^*\alpha_q = \alpha_p$. It follows that $\alpha$ is the required element of $\cl{T}^kX$ for which $\gamma_X(\alpha) = g$.
  \end{pfparts}

\end{proof}

\begin{remark}
    The existence of $\gamma$ is, in fact, a consequence of the universal property of the right Kan extension. Consider the functor
  \begin{equation*}
    G\coloneqq \mor{Dflg}((y\hat{T}_k)^\op-, \R) \colon \cat{Eucl}^\op \to \cat{Dflg}.
  \end{equation*}
  There is a natural transformation $\delta\colon \hat{\cl{T}}^k \to G$ defined by $\delta_U(\alpha)(u,v) \coloneqq \alpha(u)(v)$. It is straightforward to check that $\delta_U$ is smooth and that $\delta$ is natural in $U$. Therefore the vertical composition $\delta \circ \epsilon^{\hat{\cl{T}}^k}$ is a natural transformation $\cl{T}^k \circ y^\op \to G$. The universal property of $\Ran_{y^\op}G$ thus gives a unique natural transformation $\gamma'\colon \cl{T}^k \to \Ran_{y^\op}G$ with the property that $\delta \circ \epsilon^{\hat{\cl{T}}^k} = \epsilon^G \circ (\gamma y^{\op})$. Here $\gamma y^\op$ is the whiskering of $\gamma$ along $y^\op$. Now we compute:
\begin{align*}
  \Ran_{y^\op G}(X) &= \lim(G \circ \Pi_X) \\
                 &= \lim (\mor{Dflg}((y\hat{T}_k)^\op \circ \Pi_X-, \R) \\
                 &= \lim(\mor{Dflg}((y\hat{T}_k\circ \Pi^X)^\op - , \R) \\
                 &\cong \mor{Dflg}(\colim y\hat{T}_k \circ \Pi^X, \R) \\
  &= \mor{Dflg}(T_kX, \R).
\end{align*}
The isomorphism in the chain above is natural in $X$, so we have $\gamma'' \colon \Ran_{y^\op} G \to \mor{Dflg}(T_k-, \R)$. Then one may show that $\gamma = \gamma'' \circ \gamma'$.
\end{remark}

\subsection{Diffeological bundles of vector spaces}
\label{sec:diff-bundl-vect}

Here we give a quick review of diffeological bundles of vector spaces. For full details see \cite{ChrWu23}.

\begin{definition}
  A real \define{diffeological bundle of vector spaces} (DBV) consists of a smooth surjection $\pi\colon E \to B$, a smooth addition $+\colon E \fiber{}{B} E \to E$, and smooth fibre-preserving $\R$-multiplication $\cdot \colon \R \times E \to E$, such that all fibres $(E_b\coloneqq \pi^{-1}(b), +, \cdot)$ are real vector spaces and the zero section $B \to E$ is smooth.
\end{definition}
Each $(E_b,+,\cdot)$ is a diffeological vector space when $E_b$ is equipped with its sub-diffeology. We denote by $\Gamma(E)$ the space of smooth sections of $\pi$, equipped with the sub-diffeology it inherits from $\mor{Dflg}(B,E)$.

\begin{example}
  Let $M$ be a manifold and $D \subseteq TM$ be a smooth generalized distribution, i.e., $D_x \coloneqq \pi^{-1}(x)$ is a vector subspace of $T_xM$ for each $x$ and $D = \{X_x \mid X \in \fk{X}(M) \textrm{ and } X(M) \subseteq D\}$. Both $D$ and $TM/D$ are DBVs over $M$ with their natural addition and $\R$-multiplication, and the fibres $D_x$ and $T_xM/D_x$ carry their standard Euclidean diffeologies. The bundle $TM/D$ is called a \define{tepui} bundle in \cite{GarMiyRyv25}.  
\end{example}

\begin{definition}
  The \define{tensor product} of $\pi_1\colon E \to M$ and $\pi_2\colon F \to M$ has underlying set $E \otimes F \coloneqq \bigsqcup_{b \in B} E_b \otimes F_b$, with natural addition and scalar multiplication. A map $U \to E\otimes F$ is a plot if it is locally a finite sum $\sum_{i=1}^k  \lambda_i \cdot p_i \otimes q_i$, where the $\lambda_i \colon U \to \R$ are smooth and the $(p_i, q_i) \colon U \to E \fiber{\pi_1}{\pi_2} F$ are plots.
\end{definition}

The bundle $E \otimes F$ has the expected universal property: for any DBV $G \to B$, there is a canonical bijection between bundle morphisms $E \otimes F \to G$ and smooth fibrewise bilinear maps $E \oplus F \to G$.

\begin{definition}
  The \define{dual} of $\pi\colon E \to B$, denoted $\pi^\vee\colon E^* \to B$, has underlying set $\bigsqcup_{b \in B} E_b^*$, with natural addition and scalar multiplication. Here $E_b^*$ denotes the subspace of $\mor{Dflg}(E_b,\R)$ consisting of linear functionals $E_b \to \R$. A map $p\colon U \to E^*$ is a plot if the following maps are smooth:
  \begin{equation*}
    U \to E^* \to B, \quad r \mapsto \pi^\vee (p(u)), \textrm{ and } U \fiber{\pi^\vee \circ p}{\pi} E \to \R, \quad (r,e) \mapsto p(r)(e).
  \end{equation*}
\end{definition}

\begin{lemma}
  Let $E \to B$ and $F \to B$ be DBVs whose fibres are Euclidean spaces. The natural map associating a section of $E^*\otimes F^*$ to a smooth fibrewise bilinear map $E\otimes F \to B \times \R$ is a bijection. 
\end{lemma}
\begin{proof}
  Since the fibres of $E$ and $F$  are Euclidean, the natural map $E^*\otimes F^* \to (E\otimes F)^*$ is a bundle isomorphism. Then \cite[Proposition 3.12, Lemma 3.14]{ChrWu23} gives that the natural map $\Gamma((E_1\otimes E_2)^*) \to \mor{DBV}(E_1\otimes E_2, M \times \R)$ is an isomorphism, and we conclude by recalling that the bundle maps $E\otimes F \to M\times \R$ are canonically in bijection with the smooth fibrewise bilinear maps $E\oplus F \to M \times \R$.
\end{proof}

\printbibliography

\end{document}